\documentclass{llncs}

\usepackage{hyperref}
\usepackage{graphicx}
\usepackage{stmaryrd}
\usepackage{subfig}
\usepackage{latexsym}
\usepackage{amsmath}
\usepackage{amssymb}
\usepackage{inputenc}
\usepackage{url}
\usepackage{titletoc}
\usepackage{tabularx}

\usepackage{stmaryrd}

\newcommand{\YS}{{$\gamma$\textsf{RYS }}}

\newcommand{\YSP}{{$\gamma$\textsf{RYS+ }}}
\newcommand{\pc}{\mathbf{P}}
\newcommand{\oc}{\mathbf{O}}
\newcommand{\uc}{\mathbf{U}}
\newcommand{\pp}{\mathbb{P}}
\newcommand{\ox}{\mathbb{X}}
\newcommand{\po}{\mathbb{O}}

\newcommand{\md}{\mathrm{Md}}
\newcommand{\ms}{\mathcal{S}}

\begin{document}

\title{Dialectics of Counting and the Mathematics of Vagueness}
\author{A. Mani}
\institute{Department of Pure Mathematics\\
University of Calcutta\\
9/1B, Jatin Bagchi Road\\
Kolkata(Calcutta)-700029, India\\
\texttt{$a.mani.cms@gmail.com$}\\
Homepage: \url{http://www.logicamani.in}}

\maketitle

\begin{abstract}
New concepts of rough natural number systems are introduced in this research paper from
both formal and less formal perspectives. These are used to improve most rough
set-theoretical measures in general Rough Set theory (\textsf{RST}) and to represent rough
semantics. The foundations of the theory also rely upon the axiomatic approach to
granularity for all types of general \textsf{RST} recently developed by the present
author. The latter theory is expanded upon in this paper. It is also shown that algebraic
semantics of classical \textsf{RST} can be obtained from the developed dialectical
counting procedures. Fuzzy set theory is also shown to be representable in purely
granule-theoretic terms in the general perspective of solving the contamination problem
that pervades this research paper. All this constitutes a radically different approach to
the mathematics of vague phenomena and suggests new directions for a more realistic
extension of the foundations of mathematics of vagueness from both foundational and
application points of view. Algebras corresponding to a concept of \emph{rough naturals}
are also studied and variants are characterised in the penultimate section.

\medskip
\textbf{keywords}:
Mathematics of Vagueness, Rough Natural Number Systems, Axiomatic Theory of Granules,
Granulation, Granular Rough Semantics, Algebraic Semantics, Rough Y-Systems, Cover Based
Rough Set Theories, Rough Inclusion Functions, Measures of Knowledge, Contamination
Problem.

\end{abstract}


\section{Introduction}

Rough and Fuzzy set theories have been the dominant approaches to vagueness and
approximate reasoning from a mathematical perspective. Some related references
are
\cite{Sk08,Paw94,PPM2,Baz06,Sk05,AM909,Ba03,We07,Sh06,Ke00,Ke99,Ru23,Paw97,Bo08,DP80,LP02}
.
In rough set theory (\textsf{RST}),vague and imprecise information are dealt with through
binary relations (for some form of indiscernibility) on a set or covers of a set or
through more abstract operators. In classical \textsf{RST} \cite{ZPB}, starting from an
approximation space consisting of a pair of a
set and an equivalence relation over it, approximations of subsets of the set are
constructed out of equivalence partitions of the space (these are crisp or definite) that
are also regarded as granules in many senses. Most of the developments in \textsf{RST}
have been within the ZFC or ZF set-theoretic framework of mathematics. In such frame
works, rough sets can be seen as pairs of sets of the form $(A,\,B)$, with $A\subseteq B$
or more generally as in the approaches of the present author as collections of ''some
sense definite elements'' of the form \[\{a_1, a_2, \ldots a_{n},\,b_{1}, b_{2},\ldots b_r
\}\] subject to $a_{i}$s being 'part of' of some of the $b_{j}$s (in a Rough $Y$-system)
\cite{AM99}.  

Relative \textsf{RST}, fuzzy set theory may be regarded as a complementary approach or as
a special case of \textsf{RST} from the membership function perspective \cite{ZP5}. Hybrid
rough-fuzzy and fuzzy-rough variants have also been studied. In these a partitioning of
the meta levels can be associated for the different types of phenomena, though it can be
argued that these are essentially of a rough set theoretic nature. All of these approaches
have been confined to ZFC or ZF or the setting of classical mathematics. Exceptions to
this trend include the Lesniewski-mereology based approach \cite{PS3}. Though Rough
$Y$-systems have been introduced by the present author in ZF compatible settings
\cite{AM99}, they can be generalised to semi-sets and variants in a natural way. This
semi-set theoretical variant is work in progress. Note that the term 'theory' in
\textsf{RST} is being used in the loose sense of the literature on the subject. It matters
because the paper has strong connections with logical systems and philosophy. 

The granular computing paradigm can be traced to the mid nineties and has been used in
different fields including fuzzy and rough set theories. An overview is considered in
\cite{TYL}. The dominant interpretation of the paradigm within \textsf{RST} has been that
granularity is a context-dependent notion that gets actualized in a quasi inductive way
(see \cite{YY5} for example). A few axiomatic approaches for specific rough set theories
are known in the literature \cite{WW}, while in some others like \cite{SW3} different
types of granules have been used in tolerance approximation spaces. The axiomatic theory
\cite{AM99} developed by the present author is the most general one in the context of
general \textsf{RST}s. \cite{KCM} considers 'ontology driven formal theory of granules'
for application to biological information systems, related specifications and logic
programming. The motivations and content are mostly orthogonal to the concerns of the
present research paper.

As classical \textsf{RST} is generalised to more general relations and covers, the process
of construction and definition of approximations becomes more open ended and draws in
different amounts of arbitrariness or hidden principles. The relevant concept of
'granules', the things that generate approximations, may also become opaque in these
settings. To address these issues and for semantic compulsions, a new axiomatic theory of
granules has been developed over a \textsf{RYS} in \cite{AM99} and other recent papers by
the present author. This theory has been used to formalise different principles, including
the local clear discernibility principle in the same paper. In this research, it is
extended to various types of general \textsf{RST} and is used to define the concepts of
discernibility used in counting procedures, generalised measures and the problems of
representation of semantics.  

Many types of contexts involving vagueness cannot be handled in elegant way through
standard mathematical techniques. Much of \textsf{RST} and \textsf{FST} are not known to
be particularly elegant in handling different measures like the degree of membership or
inclusion. For example, these measures do not determine the semantics in clear terms or
the semantics do not determine the measures in a unique way. But given assumptions about
the measures, compatible semantics \cite{PL} in different forms are well known. This
situation is due to the absence of methods of counting collections of objects including
relatively indiscernible ones and methods for performing arithmetical operations on them.
However in the various algebraic semantics of classical \textsf{RST} some boundaries on
possible operations may be observed.

The process of counting a set of objects, given the restriction that some of these may be
indiscernible within themselves, may appear to be a very contextual matter and on deeper
analysis may appear to bear no easy relationship with any fine structure concerning the
vagueness of a collection of such elements or the rough semantics (algebraic or frame).
This is reflected in the absence of related developments on the relationship between the
two in the literature and is more due to the lack of developments in the former. 

It should be noted that the convenience of choice between concepts of formal logics (in
axiomatic or sequent calculus form) preceding algebraic semantics or the converse depend
on the context. For many classes of logics, the absence of real distinction is
well-known \cite{GT,FJ}. Any intention to deal with models automatically comes with an
ontological commitment to the proof-theoretical approach and vice versa. The literature on
rough sets shows that the converse works better -- for example the rough algebra semantics
\cite{BC1}, the double stone algebra semantics \cite{DU} and super rough semantics
\cite{AM3} were developed from a 'semantic viewpoint'. The modal perspective originated
from both a semantic viewpoint and an axiomatic viewpoint (not proof-theoretic). The more
important thing to be noted is that full application contexts of rough sets come with many
more kinds of problems (like the reduct problem) that are not represented in rough logics
or semantics, since the focus is on reasoning. This means fragments of the full process
are
abstracted for the formulation of rough logics in proof-theoretical or model-theoretical
terms. When I speak of \emph{semantics of \textsf{RST}}, I mean such an abstraction
necessarily. More clarifications are provided in the section on semantic domains and the
contamination problem.

In this research, theories of \emph{vague numbers} or rather procedures for counting
collections of objects including indiscernible ones have been introduced by the present
author and have been applied to extend various measures of \textsf{RST}. The extended
measures have better information content and also supplement the mereological theory from
a quantitative perspective. Proper extension of these to form a basis of \textsf{RST}
within ZF/ZFC is also considered in this research paper. Here, by a 'basis of
\textsf{RST}', I mean a theory analogous to the theory of numbers from which all
mathematics of exact phenomena can be represented. Throughout this paper, the theory may
be seen to be restricted to ZF/ZFC set theoretical setting, though a naive or a second
order reading will not be problematic. Relaxation of the ZF axioms given the dialectical
interpretation of semantic domain will be taken up in subsequent papers, but the
philosophical motivations for such a paradigm shift will be considered in later sections.
From a purely vagueness perspective, the goal is also to enlarge the scope of mathematical
perspective of vagueness. 

Notation and terminology are fixed in the second section. In the third section, the basic
orientation of object and meta levels used, and the relation with concepts is elucidated.
The concept of \emph{contamination} of information across meta-levels is introduced and
described in the next section. In the fifth section, the reason for using a fragment of
mereology as opposed to the Polkowski-Skowron mereology is explained. Some non-standard
(with respect to the literature on \textsf{RST}) examples motivating key aspects of the
axiomatic approach to granules are presented in the sixth section. In the next section,
the entire structure of the proposed program and aspects of the measures used in
\textsf{RST} are discussed. In the following section, aspects of counting in domains of
vague reasoning are explained in a novel perspective. In the ninth section, the axiomatic
theory of granules over rough $Y$-systems is extended. In the following two sections this
is applied to relation-based and cover-based rough set theories. The ninth, tenth and
eleventh sections may also be found in a forthcoming paper by the present author and have
been included for completeness. Dialectical counting processes are introduced next. These
are used to generalise rough inclusion functions, degrees of knowledge dependency and
other measures in the following section. In the fourteenth section, possible
representation of different types of counts is developed. An application to rough
semantics and integration of granularity with a method of counting is considered in the
fifteenth section. In the following section, I show how fuzzy set theory can be viewed as
a particular form of granularity in the perspective of the contamination problem. The
relation with earlier approaches is also indicated.  Subsequently I consider the problem
of improving the representation of counts in a low-level perspective and develop the
algebra of rough naturals in detail. Further directions are mentioned in the eighteenth
section.

\section{Some Background, Terminology}

A \emph{Tolerance Approximation Space} \textsf{TAS} \cite{ZP6} is a pair $S=\left\langle
\underline{S},\,T\right\rangle$, with $\underline{S}$ being a set and $T$ a tolerance
relation over it. They are also known as similarity and as tolerance approximation spaces
(conflicting the terminology introduced in \cite{SS1}). For each $x\in S$, the associated
set of $T$-related elements is given by $[x]_{T}=\{y\,;\,(x,\,y)\,\in\,T\}$. Some
references for extension of classical RST to \textsf{TAS} are \cite{SS1}, \cite{KM},
\cite{CG98} and \cite{KB98}. In \cite{SW3} specific granulations are considered separately
in \textsf{TAS}, but many types of duality and connections with logics are not considered.
The actual body of work in the field is huge and no attempt to mention all possibly
relevant references will be made.

An approach \cite{KM} has been to define a new equivalence $\theta_{0}$ on  $S$ via
$(x,\,y)\,\in\,\theta_{0}$ if and only if $dom_{T}(x)=dom_{T}(y)$  with 
$dom_{T}(z)=\cap\{[x]_{T}\,:\,z\,\in\,[x]_{T}\}$. This is an unduly cautious 'clear
perspective' approach. A generalization of the approximation space semantics using
$T$-related sets (or tolerance sets) can be described from the point of view of
generalised covers (see \cite{IM}). This includes the approach to defining the lower and
upper approximation of a set $A$ as
\[A^{l}=\bigcup\{[x]_{T}\,;\,[x]_{T}\,\subseteq\,A\},\] and
\[A^{u}=\bigcup\{[x]_{T}\,;\,[x]_{T}\,\cap\,A\,\neq\,\emptyset ,\,x\,\in\,A\}.\] A
\emph{bited modification} proposed in \cite{SW}, valid for many definable concepts of
granules, consists in defining a \emph{bited upper approximation}. Algebraic semantics of
the same has been considered by the present author in \cite{AM105}. It is also shown that
a full representation theorem is not always possible for the semantics.    

The approximations \[A^{l*}=\{
x\,;\,(\exists{y})\,(x,\,y)\,\in\,T,\,[y]_{T}\,\subseteq\,A\},\] and
\[A^{u*}=\{x\,;\,(\forall{y})\,((x,\,y)\,\in\,T\,\longrightarrow\,[y]_{T}\,\cap\,A\,\neq\,
\emptyset)\} = (A^{c})^{l*c}\] were considered in \cite{PO,CG98}. It can be shown that,
for any subset $A$,
\[A^{l}\subseteq\,A^{l*}\,\subseteq\,A\,\subseteq\,A^{u*}\,\subseteq\,A^{u}.\] In the BZ
and Quasi-BZ algebraic semantics \cite{CC}, the lower and upper rough operators are
generated by a preclusivity operator and the complementation relation on the power set of
the approximation space, or on a collection of sets under suitable constraints in a more
abstract setting. Semantically, the BZ-algebra and variants do not capture all the
possible ways of arriving at concepts of discernibility over similarity spaces. 

Let ${S}$ be a set and $\mathcal{S}\,=\,\{K_{i}\}_{1}^{n}\,:n\,<\,\infty$ be a collection
of subsets of it. We will abbreviate subsets of natural numbers of the form $\{1, 2,
\ldots , n\}$ by $\mathbf{N}(n)$. For convenience, we will assume that
$K_{0}\,=\,\emptyset$, $K_{n+1}\,=\,S$. $\left\langle S,\,\mathcal{S} \right\rangle $ will
also be referred to as a \emph{Cover Approximation System} (CAS).

Cover-based \textsf{RST} can be traced to \cite{WZ}, where the approximations $A^{l}$ and
$A^{u}$  are defined over the cover $\{[x]_{T}; x\in S\}$. A 1-neighbourhood \cite{YY9}
$n(x)$ of an element $x\in S$ is simply a subset of $S$. The collection of all
1-neighbourhoods $\mathcal{N}$ of $S$ will form a cover if and only if $(\forall
x)(\exists y) x\in n(y)$ (anti-seriality). So in particular a reflexive relation on $S$ is
sufficient to generate a cover on it. Of course, the converse association does not
necessarily happen in a unique way.     

If $\mathcal{S}$ is a cover of the set $S$, then the \emph{Neighbourhood} \cite{LTJ} of
$x\in S$ is defined via, \[nbd(x)\,=\,\bigcap\{K:\,x\in K\,\in\,\mathcal{S}\}.\] The sixth
type of lower and upper approximations \cite{ZW3,YY9} of a set $X$ are then defined by
\[X_{\$}\,=\,\{x:\,nbd(x)\,\subseteq\,X\},\] and
\[X^{\$}\,=\,\{x:\,nbd(x)\cap\,X\neq\,\emptyset \}.\]

The minimal description of an element $x\in S$ is defined to be the collection
\[\mathrm{Md} (x)\,=\,\{A\,:x\in A\in \mathcal{S},\, \forall{B}(x\in B\rightarrow
\sim(A\subset B))\}.\] The \emph{Indiscernibility} (or friends) of an element $x\,\in S$
is defined to be \[Fr(x)\,=\,\bigcup\{K:\,x\in K\in\mathcal{S}\}.\] The definition was
used first in \cite{PO}, but has been redefined again by many others (see \cite{CP4}). An
element $K\in \ms$ will be said to be \emph{Reducible} if and only if \[(\forall x\in K)
K\neq\,\md(x).\] The collection $\{nbd(x):\,x\in\,S\}$ will be denoted by $\mathcal{N}$.
The cover obtained by the removal of all reducible elements is called a covering reduct.
The terminology is closest to \cite{ZW3} and many variants can be found in the literature
(see \cite{CP4}). 

If $X\subseteq  S$, then let
\begin{enumerate} \renewcommand\labelenumi{\theenumi}
  \renewcommand{\theenumi}{(\roman{enumi})}
\item {$X^{l1}\,=\,\bigcup\{K_{i}\,:\,K_{i}\subseteq  X,\,\,i\,\in \,\{0,\,1,\,...,n\} 
\}$.} 
\item {$X^{l2}\,=\,\bigcup\{\cap_{i\,\in \,I}(S\,\setminus K_{i})\,:\,\cap_{i\,\in
\,I}(S\,\setminus\, K_{i})\subseteq X,\,\,I\subseteq \mathbf{N}(n+1) \} $; the union is
over the $I$'s.}
\item {$X^{u1}\,=\,\bigcap\{\cup_{i\,\in \,I}{K_{i}}\,:\,X\subseteq \cup_{i\,\in
\,I}\,K_{i},\,\,I\subseteq \mathbf{N}(n+1) \} $; the intersection is over the $I$'s.}
\item {$X^{u2}\,=\,\bigcap\{S\,\setminus\,{K_{i}}\,:\,X\subseteq
S\,\setminus\,K_{i},\,\,i\,\in \,\{0,\,...,n\}\}$.}
\end{enumerate}

The pair $(X^{l1},\,X^{u1})$ is called an $AU$-\emph{rough set} by union, while
$(X^{l2},\,X^{u2})$ an $AI$-\emph{rough set} by intersection (in the present author's
notation \cite{AM24}). In the notation of ~\cite{IM}, these are $\left(
\mathcal{F}_{*}^{\cup}(X),\,\mathcal{F}_{\cup}^{*}(X)\right)$ and $\left(
\mathcal{F}_{*}^{\cap}(X),\,\mathcal{F}_{\cap}^{*}(X)\right)$, respectively. I will also
refer to the pair $\left\langle S,\,\mathcal{K} \right\rangle $ as an
\textsf{AUAI}-\emph{approximation system}.

\begin{theorem}
The following hold in AUAI approximation systems:

\begin{enumerate}\renewcommand\labelenumi{\theenumi}
  \renewcommand{\theenumi}{(\roman{enumi})} 
\item {$X^{l1}\subseteq \,X\subseteq X^{u1}$;  $X^{l2}\subseteq \,X\subseteq X^{u2} $; 
$\emptyset^{l1}\,=\,\emptyset^{l2}\,=\,\emptyset $,}
\item {$(\cup{\mathcal{K}}\,=\,S\,\longrightarrow\,S^{u1}\,=\,S^{u2}\,=\,S) $ ; 
$(\cup{\mathcal{K}}\,=\,S\,\longrightarrow\,\emptyset^{u2}\,=\,\emptyset,\,S^{l1}\,=\,S)
$,}
\item
{$(\cap{\mathcal{K}}\,=\,\emptyset\,\longrightarrow\,\emptyset^{u1}\,=\,\emptyset,\,S^{l2}
\,=\,S) $,}
\item {$(X\,\cap\,Y)^{l1}\subseteq X^{l1}\,\cap\,Y^{l1}$ ; 
$(X\,\cap\,Y)^{l2}\,=\,X^{l2}\,\cap\,Y^{l2}$,}
\item {$(X\,\cup\,Y)^{u1}\,=\,X^{u1}\,\cup\,Y^{u1}$ ;    $X^{u2}\,\cup\,Y^{u2}\subseteq
(X\,\cup\,Y)^{u2}$,}
\item {$(X\subseteq Y\,\longrightarrow\,X^{l1}\subseteq Y^{l1},\,X^{l2}\subseteq Y^{l2},
X^{u1}\subseteq Y^{u1},\,X^{u2}\subseteq Y^{u2})$,}
\item {If $(\forall i \neq j) K_{i}\cap K_{j} = \emptyset$ then $(X\,\cap\,Y)^{l1} =
X^{l1}\,\cap\,Y^{l1},\, (X\,\cup\,Y)^{u2} = X^{u2}\,\cup\,Y^{u2} $,}
\item {$X^{l1}\,\cup\,Y^{l1}\subseteq (X\,\cup\,Y)^{l1}$ ;  $X^{l2}\,\cup\,Y^{l2}\subseteq
(X\,\cup\,Y)^{l2}$,}
\item {$(X\,\cap\,Y)^{u1}\subseteq X^{u1}\,\cap\,Y^{u1}$ ;  $(X\,\cap\,Y)^{u2}\subseteq
X^{u2}\,\cap\,Y^{u2}$,}
\item {$(S\,\setminus\,X)^{l1}\,=\,S\,\setminus\,X^{u2}$ ; 
$(S\,\setminus\,X)^{l2}\,=\,S\,\setminus\,X^{u1}$, } 
\item {$(S\,\setminus\,X)^{u1}\,=\,S\,\setminus\,X^{l2}$ ; 
$(S\,\setminus\,X)^{u2}\,=\,S\,\setminus\,X^{l1}$,}
\item {$(X^{l1})^{l1}\,=\,X^{l1}$ ; $(X^{l2})^{l2}\,=\,X^{l2}$ ;
$(X^{u1})^{u1}\,=\,X^{u1}$, }
\item {$(X^{u2})^{u2}\,=\,X^{u2}$ ; $(X^{l1})^{u1}\,=\,X^{l1}$ ;  
$(X^{u2})^{l2}\,=\,X^{u2}$,}
\item {$X^{l2}\subseteq (X^{l2})^{u2}$,   $(X^{u1})^{l1}\subseteq X^{u1}$, }
\item {$(\mathcal{K}_{j}^{\cap}(X))^{u2}\,=\,\mathcal{K}_{j}^{\cap}(X)$, $ j=
1,\,2,...,t_{1}$ ; $(\mathcal{K}_{j}^{\cup}(X))^{l1}\,=\,\mathcal{K}_{j}^{\cup}(X)$, $j =
1,\,2,...,t_{2}$.}
\end{enumerate}
In this, $(\mathcal{K}_{j}^{\cup}(X))$ is the minimal union of sets of the form $K_{i}$
that include $X$ (for $j$ being in the indicated range) and $(\mathcal{K}_{j}^{\cap}(X))$
is the maximal intersection of sets of the form $K_{i}^{c}$ that are included in $X$.
\end{theorem}

All of the above concepts can be extended to covers with an arbitrary number of elements.
The concepts of indiscernibility, neighbourhood and minimum description can be extended to
subsets of $S$. The concept of a \emph{Neighbourhood Operator} has been used in the
literature in many different senses. These can be relevant in the context of the sixth
type ($l6+,\,u6+$) (see the sixth section) approximations for dealing with covers
generated by partially reflexive relations \cite{AM24}. A large number of approximations
in the cover-based approximation context have been studied in the literature using a far
larger set of notations. An improved nomenclature is also proposed in the eleventh
section.

Cover-based \textsf{RST} is more general than relation-based \textsf{RST} and the question
of when covers over a set correspond to relations over the set is resolved through duality
results. It is well known that partitions correspond to equivalences and normal covers to
tolerances. The approach based on neighbourhoods \cite{YY9} provides many one way results.
A more effective way of reducing cover-based \textsf{RST} to relation-based \textsf{RST}
is in \cite{AM960}.

\section{Semantic Domains, Meta and Object Levels}

This section is intended to help with the understanding of the section on the
contamination problem, the definition of \textsf{RYS} and clarify the terminology about
meta and object levels among other things. In classical \textsf{RST} (see \cite{ZPB}), an
approximation space is a pair of the form $\left\langle S,\,R \right\rangle $, with $R$
being an equivalence on the set $S$. On the power set $\wp (S)$, lower and upper
approximation operators, apart from the usual Boolean operations, are definable. The
resulting structure constitutes a semantics for RST (though not satisfactory) from a
classical perspective. This continues to be true even when $R$ is some other type of
binary relation. More generally (see fourth section) it is possible to replace $\wp (S)$
by some set with a parthood relation and some approximation operators defined on it. The
associated semantic domain in the sense of a collection of restrictions on possible
objects, predicates, constants, functions and low level operations on those will be
referred to as the classical semantic domain for general RST. In contrast, the semantics
associated with sets of roughly equivalent or relatively indiscernible objects with
respect to this domain will be called the rough semantic domain. Actually many other
semantic domains, including hybrid semantic domains, can be generated (see \cite{AM699},
\cite{AM105} \cite{AM3}) for different types of rough semantics, but these two broad
domains will always be - though not necessarily with a nice correspondence between the
two. In one of the semantics developed in \cite{AM105}, the reasoning is within the power
set of the set of possible order-compatible partitions of the set of roughly equivalent
elements.  The concept of \emph{semantic domain} is therefore similar to the sense in
which it is used in general abstract model theory \cite{MD} (though one can object to
formalisation on different philosophical grounds).  

Formal versions of these types of semantic domains will be useful for clarifying the
relation with categorical approaches to fragments of granular computing \cite{BY}. But
even without a formal definition, it can be seen that the two approaches are not
equivalent. Since the categorical approach requires complete description of fixed type of
granulations, it is difficult to apply and especially when granules evolve relative
particular semantics or semantic domains. The entire category \textbf{ROUGH} of rough sets
in \cite{BY}, assumes a uniform semantic domain as evidenced by the notion of objects and
morphisms used therein. A unifying semantic domain may not also be definable for many sets
of semantic domains in our approach. This means the categorical approach needs to be
extended to provide a possibly comparable setting.

The term \emph{object level} will mean a description that can be used to directly
interface with fragments (sufficient for the theories or observations under consideration)
of the concrete real world. Meta levels concern fragments of theories that address aspects
of dynamics at lower meta levels or the object level. Importantly, we permit meta level
aspects to filter down to object levels relative different object levels of specification.
So it is always possible to construct more meta levels and expressions carry intentions. 

\emph{Despite all this, two particular meta levels namely Meta-C (or Meta Classical),
Meta-R (or Meta Rough) and an object level will be used for comparing key notions
introduced with the more common approaches in the literature. Meta-R is the meta level
corresponding to the observer or agent experiencing the vagueness or reasoning in vague
terms (but without rough inclusion functions and degrees of inclusion), while Meta-C will
be the more usual higher order classical meta level at which the semantics is formulated.
It should be noted that rough membership functions and similar measures are defined at
Meta-C, but they do not exist at Meta-R. A number of meta levels placed between Meta-R and
Meta-C can possibly be defined and some of these will be implicit in the section on rough
naturals.} 
 
Many logics have been developed with the intent of formalising 'rough sets' as
'well-formed formulae' in a fixed language. They do not have a uniform domain of discourse
and even ones with category theoretically equivalent models do not necessarily see the
domain in the same way (though most meanings can be mapped in a bijective sense). For
example, the regular double stone algebra semantics and complete rough algebra semantics
correspond to different logical systems of classical \textsf{RST} (see \cite{BK3,BC2}).
The super rough algebra semantics in \cite{AM3} actually adds more to the rough algebra
semantics of \cite{BC1}. It is possible to express the ability of objects to approximate
in the former, while this is not possible in the latter. This is the result of a higher
order construction used for generating the former.    

The relation of some rough semantics and topology mentioned in the previous section is
again a statement about the orientation of the semantic domains in the respective subjects
formulated in more crude mathematical terms. 

\subsection{Granules and Concepts}

In \cite{YY3} for example, concepts of human knowledge are taken to consist of an
intensional part and an extensional part. The \emph{intension} of a concept is the
collection of properties or attributes that hold for the set of objects to which the
concept applies. The \emph{extension} is to consist of actual examples of the object. Yao
writes, 'This formulation enables us to study concepts in a logic setting in terms of
intensions and also in a set-theoretic setting in terms of extensions'. The description of
granules characterise concepts from the intensional point of view, while granules
themselves characterise concepts from the extensional point of view. \emph{Granulations}
are collections of granules that contain every object of the universe in question. In a
seemingly similar context, in \cite{PPM} (or \cite{PPM2}) the authors speak of
\emph{extensional granules} and \emph{intensional granules} that are respectively related
to objects and properties. In my opinion the semantic domains in use are different and
these are not conflicting notions, though it is equally well to call the latter a more
strong platonic standpoint. Yao does not take sides on the debate in \emph{what a concept
is} and most of it is certainly nonclassical and non empiricist from a philosophical
point of view. 

In modern western philosophy, intentions and extensions are taken to be possessed by
linguistic expressions and not by concepts. Thus, for example, from Frege's point of view,
the intension is the concept expressed by the expression, and the extension is the
collection of items to which the expression applies. In this perspective, the concept
applies to the same collection of items. It also follows that concepts, in this
perspective, must be tied to \emph{linguistic expressions} as well.  

Concepts are constituents of thinking containing the meaning of words or intended action
or response. As such a linguistic expression for such \emph{concepts} may not be supplied
by the reasoner. Apparently the Fregean point of view speaks of concepts with associated
linguistic expression alone. Even if we use a broad-sense notion of 'linguistic
expression', this may fall short of the \emph{concept} mentioned in the former viewpoint.
Another key difference is that the former version of \emph{concepts} are bound to be more
independent of interpreters (or agents) than the latter. The concept of \emph{granules}
actually evolves across the temporal space of the theory and may be essentially \emph{a
priori } or \emph{a posteriori} (relative to the theory or semantics) in nature. Because
of
these reasons, I will not try to map the two \emph{concepts} into each other in this paper
at least. In the present paper, \emph{a priori} granules will be required in an essential
way. 

It is only natural that possible concepts of granules are dependent on the choice of
semantic domain in the contexts of \textsf{RST}. But \emph{a priori} granules may even be
identified at some stage after the identification of approximations.

\section{Contamination Problem}

Suppose the problem at hand is to model vague reasoning in a particular context and
relative to the agents involved in the context. It is natural for the model to become
\emph{contaminated} with additional inputs from a classical perspective imposed on
the context by the person doing the modelling. In other words, meta level aspects can
contaminate the perception of object level features. From an algebraic perspective, if the
model concerns objects of specific types like 'roughly equivalent objects in some sense',
then the situation is relatively better than a model that involves all types of objects.
But the operations used in the algebra or algebraic system can still be viewed with
suspicion. 

By the contamination problem, I mean the problem of minimising or eliminating the
influences of the classicist perspective imposed on the context. In other words, the
problem is to minimize the contamination of models of meta-R fragments by meta-C
aspects.  One aspect of the problem is solved partially in \cite{AM1005} by the present
author. In the paper, a more realistic conception of rough membership functions and other
measures of \textsf{RST} have been advanced from a minimalist perspective avoiding the
real-valued or rational-valued rough measures that dominate the rough literature. Most of
the rough measures based on cardinalities are of course known to lack mathematical rigour
and have the potential to distort the analysis. 

In the mathematics of exact phenomena, the natural numbers arise in the way they do
precisely because it is assumed that things being counted are well-defined and have exact
existence. When a concrete collection of identical balls on a table are being counted,
then it is their relative position on the table that helps in the process. But there are
situations in real life, where 
\begin{itemize}\renewcommand{\labelitemi}{$\bullet$}
\item {such identification may not be feasible,}
\item {the number assigned to one may be forgotten while counting subsequent objects,}
\item {the concept of identification by way of attributes may not be stable,}
\item {the entire process of counting may be 'lazy' in some sense,}
\item {the mereology necessary for counting may be insufficient.}
\end{itemize}

Apart from examples in \cite{AM1005}, the most glaring examples for avoiding the measures
comes from attempts to apply rough sets to modelling the development of human concepts.
The 'same'
concept $X$ may depend on ten other concepts in one perspective and nine other concepts in
another perspective and concepts of knowing the concept $X$ and gradation does not admit a
linear measure in general. Using one in fields like education or education research would
only encourage scepticism. The quality of measures like 'impact factor' of journals
\cite{AF} provide a supportive example.     

The underlying assumptions behind rough measures are much less than in a corresponding
statistical approach (subject to being possible in the first place in the application
context in question) and do not make presumptions of the form -'relative errors should
have some regularity'. Still the contamination problem is relevant in other domains of
application of \textsf{RST} and more so when the context is developed enough to permit an
evaluation of semantic levels. 

There may be differences in the semantic approach of proceeding from algebraic models to
logics in sequent calculus form in comparison to the approach of directly forming the
logic as a sequent calculus, or the approach of forming the logic in Kripke-like or
Frame-related terminology, but one can expect one to feed the other. It should also be
noted that this has nothing to do with supervaluationary perspectives \cite{Fine75}, where
the goal is to reduce vagueness by improving the language. Moreover the primary concerns
in the contamination problem are not truth-values or gaps in them. The contamination
problem is analogous to the problem of intuitionist philosophers to find a perfect logic
free from problematic classicist assumptions. A difficult approach to the latter problem
can be found in \cite{Dummet}. The important thing to note in \cite{Dummet} is the
suggestion that it is better to reason within a natural deduction system to generate 'pure
logic'. In case of the contamination problem, general understanding stems from
model theoretic interpretations and so should be more appropriate.         

If a model-theoretic perspective is better, then it should be expected to provide
justification for the problem. The problems happen most explicitly in attempts to model
human reasoning, in conceptual modelling especially (in learning contexts), in attempts to
model counting processes in the presence of vagueness and others. In applications to
machine intelligence, an expression of contamination would be 'are you blaming machines
without reason?'

\section{Formalism Compatibility and Mereology}

In the literature various types of mereologies or theories of part-whole relationships
\cite{PSI,LDK} are known. For the axiomatic theory, I used a minimal fragment derived from
set-theory compatible mereology in \cite{AV}. This fragment may also be argued to  be
compatible with even Lesniewskian mereology - but such arguments must be founded on scant
regard of Lesniewski's nominalism and the distortions of his ideas by later authors. Such
distortion is used as the base for generalization in all of 'Lesniewski ontology-based
rough mereology'. This is evident for example, from section-2.1 of \cite{LP4}. Even the
perspective of Gregoryck \cite{GK} is accepted - 'theorems of ontology are those that are
true in every model for atomic Boolean algebras without a null element'. Other papers by
the same author, confirm the pragmatic excesses as classical rough set theory is shown to
be embedded in the rough mereological generalisation as well \cite{PLS}. New
problems/conflicts of a logical/philosophical nature are bound to crop up if the theory is
applied to model reasoning and attempts are made to link it to Lesniewski's approach. In
my opinion, it would be better to term the Polkowski-Skowron approach a
'Lesniewski-inspired' mereology rather than a 'Lesniewski-ontology-based' one.     

The reader can find a reasonable re-evaluation of formal aspects of the mereology of
Lesniewski from a 'platonized perspective' in \cite{UR}. Importantly, it highlights
difficulties in making the formalism compatible with the more common languages of modern
first order or second order logic. The correct translation of expression in the  language
of ontology to a common language of set theory (modulo simplifying assumptions) requires
many axiom schemas and the converse translation (modulo simplifying assumptions) is
doubtful (see pp 184-194,\cite{UR}). 
I am stressing this because it also suggests that the foundational aspects of \cite{PS3}
should be investigated in greater detail in relation to:

\begin{itemize}\renewcommand{\labelitemi}{$\bullet$}
 \item {the apparent stability of the theory in related application contexts, and}
 \item {the exact role of the rough parthood relation and role of degrees of membership
and t-norms in diluting logical categories.}
\end{itemize}
I will not go into detailed discussion of the philosophical aspects of the points made
above
since it would be too much of a deviation for the present paper.  

One of the eventual goals of the present approach is also to extend general \textsf{RST}
to semi-set theoretical contexts \cite{VP9,VH} (or modifications thereof). Semiset theory
has been in development since the 1960s and its original goals have been to capture
vagueness and uncertainty, to be clear about what exactly is available for reasoning, to
understand the infinite as 'something non-transparent', to impose a more sensible
constraint on the relation between objects and properties and to require that any grouping
or association is actualized (i.e. available at our disposal). It can be formalised as
a conservative extension of ZFC, but irrespective of this aspect, the philosophical
framework can be exploited in other directions. However, it is obviously incompatible with
the Lesniewski ontology and nominalism. This is another reason for using a fragment of
set-theoretically compatible mereology in the definition of a general rough Y-system. In
this paper, I will continue to do the theory over ZFC-compatible settings, since most of
the
present paper will be relevant for all types of rough theorists.      

In summary, the differences with the Polkowski-Skowron style mereological approach are:

\begin{enumerate} \renewcommand\labelenumi{\theenumi}
  \renewcommand{\theenumi}{(\roman{enumi})}
\item {The mereology is obviously distinct. The present approach is compatible with
Godel-Bernays classes.}
\item {No assumptions are made about the degree of inclusion or of '$x$ being an
ingredient of $y$ to a degree $r$'.}
\item {Concepts of degree are expected to arise from properties of granules and 'natural
ways' of counting collections of discernible and indiscernible objects.}
\end{enumerate}

\section{Motivating Examples for \textsf{RYS}}

Motivating examples for the general concept of \textsf{RYS} introduced in \cite{AM99} are
provided in this section. These examples do not explicitly use information or decision
tables though all of the information or decision table examples used in various
\textsf{RSTs} would be naturally relevant for the general theory developed. Other general
approaches like that of rough orders \cite{IT2} and abstract approximation spaces
\cite{CC5} are not intended for directly intercepting approximations as envisaged in
possible applications. They would also be restrictive and problematic from the point of
view of the contamination problem. Here the focus is on demonstrating the
need for many approximation operators, choice among granules and conflict resolution.    

\begin{flushleft}
\textbf{Example-1:}\\ 
\end{flushleft}

Consider the following statements associable with the description of an apple in a plate
on a table:

\begin{enumerate} \renewcommand\labelenumi{\theenumi}
  \renewcommand{\theenumi}{(\roman{enumi})}
\item {Object is apple-shaped; Object has maroon colour,}
\item {Object has vitreous skin; Object has shiny skin,}
\item {Object has vitreous, smooth and shiny skin,}
\item {Green apples are not sweet to taste,}
\item {Object does not have coarse skin as some apples do,}
\item {Apple is of variety A; Apple is of variety X.}
\end{enumerate}

Some of the individual statements like those about shape, colour and nature of skin may be
'atomic' in the sense that a more subtle characterization may not be available. It is also
obvious that only some subsets of these statements may consistently apply to the apple on
the table. This leads to the question of selecting some atomic statements over others. But
this operation is determined by consistency of all the choices made. Therefore, from a
\textsf{RST} perspective, the atomic statements may be seen as granules and then it would
also seem that choice among sets of granules is natural. More generally 'consistency' may
be replaced by more involved criteria that are determined by the goals. A nice way to see
this would be to look at the problem of discerning the object in different contexts -
discernibility of apples on trees require different kind of subsets of granules.

\begin{flushleft}
\textbf{Example-2:}\\ 
\end{flushleft}

In the literature on educational research \cite{PD} it is known that even pre-school going
children have access to powerful mathematical ideas in a form. A clear description of such
ideas is however not known and researchers tend to approximate them through subjective
considerations. For example, consider the following example from \cite{PD}:\\

Four-year-old Jessica is standing at the bottom of a small rise in the preschool yard when
she is asked by another four-year-old on the top of the rise to come up to her.\\
\begin{itemize}\renewcommand{\labelitemi}{$\bullet$}
 \item {“No, you climb down here. It’s much shorter for you.”}
\end{itemize}

The authors claim that ''Jessica has adopted a developing concept of comparison of length
to solve —
at least for her - the physical dilemma of having to walk up the rise''. But a number of
other concepts like 'awareness of the effects of gravitational field', 'climbing up is
hard', 'climbing up is harder than climbing down', 'climbing down is easier', 'climbing up
is harder', 'others will find distances shorter', 'make others do the hard work' may or
may not be supplemented by linguistic hedges like \emph{developing} or \emph{developed}
and assigned to Jessica. The well known concept of concept maps cannot be used to
visualise these considerations, because the concept under consideration is not well
defined. Of these concepts some will be assuming more and some less than the actual
concept used and some will be closer than others to the actual concept used. Some of the
proposals may be conflicting, and that can be a problem with most approaches of
\textsf{RST} and fuzzy variants. The question of one concept being closer than another may
also be dependent on external observers. For example, how do 'climbing up is harder' and
'climbing up is harder than climbing down' compare?

The point is that it makes sense to:

\begin{enumerate} \renewcommand\labelenumi{\theenumi}
  \renewcommand{\theenumi}{(\roman{enumi})}
\item {accommodate multiple concepts of approximation,}
\item {assume that subsets of granules may be associated with each of these
approximations,
}
\item {assume that disputes on 'admissible approximations' can be resolved by admitting
more approximations.}
\end{enumerate}

It is these considerations and the actual reality of different \textsf{RST} that motivates
the definition of Rough $Y$-systems.  

\section{Objectivity of Measures and General RST}

In \textsf{RST}, different types of rough membership and inclusion functions are defined
using cardinality of associated sets. When fuzzy aspects are involved then these types of
functions become more controversial as they tend to depend on the judgement of the user in
quantifying linguistic hedges and other aspects. These types of functions are also
dependent on the way in which the evolution of semantics is viewed. But even without
regard to this evolution, the computation of rough inclusion functions invariably requires
one to look at things from a higher meta level - from where all objects appear exact. In
other words an estimate of a perfect granular world is also necessary for such
computations and reasoning.

Eventually, this leads to mix up (contamination) of information obtained from perceiving
things at different levels of granularity. I find all this objectionable from the point of
view of rigour and specific applications too. To be fair such a mix up seems to work fine
without much problems in many applications. But that is due to the scope of the
applications and the fact that oversimplifications through the use of cardinality permits
a second level of 'intuitive approximation'. 

In applications of \textsf{RST} that start from information or decision tables including
information about values associated with attributes (and possibly decisions) for different
objects, the evolution of the theory adheres to the following dependency schemas: 

\begin{center}
\includegraphics[width=12.12cm]{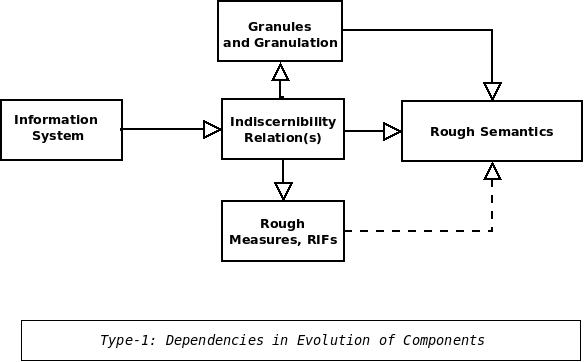} 
\end{center}

\begin{center}
\includegraphics[width=12.12cm]{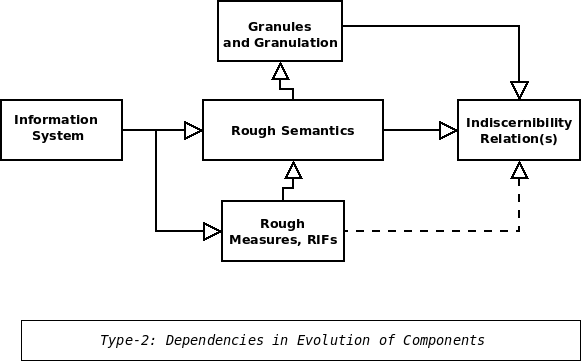} 
\end{center}

In the above two figures, 'rough semantics' can be understood to be in algebraic or in
proof-theoretic sense. The intended reading is - 'components at arrow heads depend on
those at the tail' and multiple directed paths suggest that 'components in alternate paths
may be seen in weaker forms relatively'.  These figures do not show the modified
information system that invariably results due to the computation of reducts of different
kinds, as the entire picture merely gets refreshed to the refined scenario. The
Lesniewski-style ontology-based mereological approach of \cite{PS3,PL} fits into type-1
schemas. Rule discovery approaches would fall within type-2 schemas.

The approach of the present paper is aimed at using measures that are more natural in the
rough setting and to use fewer assumptions about their evolution at the meta level.
Eventually this is likely to result in changes on methods of reduct computation in many
cases. The theory is also aimed at tackling the so-called \emph{inverse problems} of
\cite{AM3} and later papers, which is essentially 'Given a collection of definite objects
and objects of relatively less definite objects with some concepts of approximations (that
is result of vague and possibly biased considerations), find a description of the context
from a general rough perspective'. From a semantic perspective these may reduce to
abstract representation problems. The following dependency schema shows how the different
parts fit in.

\begin{center}
\includegraphics[width=12.12cm]{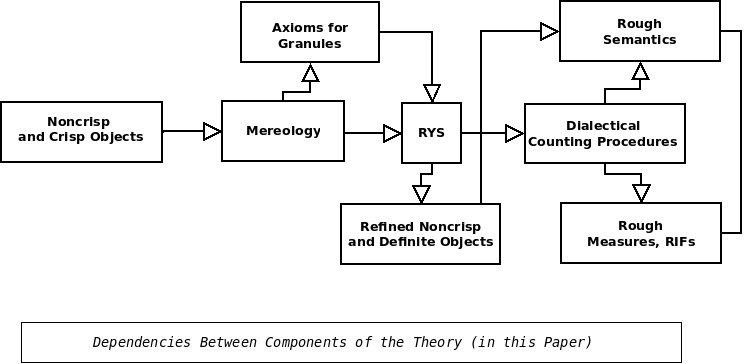} 
\end{center}

The link between 'Rough Semantics' and 'Rough Measures' should be read as 'possibly
related'. 

\section{Numbers and their Generalization}

The problems with using natural numbers for counting collections of objects including
indiscernibles have been mentioned in the fourth section. It was pointed out that there
are situations in real life, where 
\begin{enumerate} \renewcommand\labelenumi{\theenumi}
  \renewcommand{\theenumi}{(\roman{enumi})}
\item {the discernibility required for normal counting may not be feasible,}
\item {the number assigned to one may be forgotten while counting subsequent objects,}
\item {the concept of identification by way of attributes may not be stable,}
\item {the entire process of counting may be 'lazy' in some sense,}
\item {the mereology necessary for counting may be insufficient.}
\end{enumerate}

Some specific examples of such contexts are:
\begin{enumerate}
\item {Direct counting of fishes in a lake is difficult and the sampling techniques used
to estimate the population of fishes do not go beyond counting samples and pre-sampling
procedures. For example some fishes may be caught, marked and then put back into the lake.
Random samples
may be drawn from the mixed population to estimate the whole population using proportion
related statistics. The whole procedure however does not involve any actual counting of
the population.}
\item {In crowd management procedures, it is not necessary to have exact information about
the actual counts of the people in crowds.}
\item {In many counting procedures, the outcome of the last count (that is the total
number of things) may alone be of interest. This way of counting is known to be sufficient
for many apparently exact mathematical applications.}
\item {Suppose large baskets containing many varieties of a fruit are mixed together and
suppose an observer with less-than-sufficient skills in classifying the fruits tries to
count the number of fruits of a variety.  The problem of the observer can be interpreted
in mereological terms.}
\item {Partial algebras are very natural in a wide variety of mathematics. For example, in
semigroup theory the set of idempotents can be endowed with a powerful partial algebraic
structure. Many partial operations may be seen to originate from failure of standard
counting procedures in specific contexts.}
\end{enumerate}

Various generalizations of the concept of natural numbers, integers and real numbers are
known in mathematics. These generalizations are known to arise from algebraic, topological
or mixed considerations. For example, a vast amount of ring and semigroup theory arises
from properties of the integers. These include Euclidean Rings, UFD, Integral Domains,
Positively totally ordered Semigroups \cite{SM}, and Totally Ordered Commutative
Semigroups.
Partial Well Orders and Variants thereof \cite{AM90} and Difference orders can also be
seen in a similar perspective. In all these cases none of the above mentioned aspects can
be captured in any obvious way and neither have they been the motivation for their
evolution. Their actual motivations can be traced to concrete examples of subsets of real
numbers and higher order structures derived from real numbers having properties defining
the structures. Further structures like these often possess properties quite atypical of
integers.  

In counting collections of objects including relatively exact and indiscernible objects,
the situation is far more complex - the first thing to be modified would be the relative
orientation of the object and different meta levels as counting in any sense would be from
a higher meta level. Subsequently the concept of counting (as something to be realised
through injective maps into $N$) can be modified in a suitable way. The eventual goal of
such procedures should be the attainment of order-independent representations.

Though not usually presented in the form, studies of group actions, finite and infinite
permutation groups and related automorphisms and endomorphisms can throw light on lower
level counting. In the mathematics of exact phenomena, these aspects would seem
superfluous because cardinality is not dependent on the order of counting. 
But in the context of the present generalization somewhat related procedures are seen to
be usable for improving representation. A more direct algebra of Meta-R counts is also
developed in the penultimate section. They can be regarded as a natural generalization of
the ordered integral domain associated with integers and was not considered in
\cite{AM1005} by the present author. The former approach does have feasibility issues
associated. For one thing a string of relatively discernible and indiscernible things may
not be countable in all possible ways in actual practice. The latter approach takes a more
holistic view and so the two approaches can be expected to complement each other.    

\section{Granules: An Axiomatic Approach}

Different formal definitions of granules have been used in the literature on rough sets
and in granular computing. An improved version of the axiomatic theory of granules
introduced in \cite{AM99} is presented here. The axiomatic theory is  capable of handling
most contexts and is intended to permit relaxation of set-theoretic axioms at a later
stage. The axioms are considered in the framework of Rough Y-Systems mentioned earlier.
\textsf{RYS} maybe seen as a generalised form of \emph{rough orders} \cite{IT2},
\emph{abstract approximation spaces} \cite{CC5} and approximation framework \cite{CD3}. It
includes relation-based RST, cover-based RST and more. These structures are provided with
enough structure so that a classical semantic domain and at least one rough semantic
domain of roughly equivalent objects along with admissible operations and predicates are
associable. 

Within the domain of naive set theory or ZFC or second order ZFC, the approximation
framework of \cite{CD3} is not general enough because:
\begin{enumerate} \renewcommand\labelenumi{\theenumi}
  \renewcommand{\theenumi}{(\roman{enumi})}
 \item {It assumes too much of order structure.}
 \item {Assumes the existence of a De Morgan negation.}
 \item {It may not be compatible with the formulations aimed at inverse problems. That
holds even if a flexible notion of equality is provided in the language.}
\end{enumerate}

An application to the context of example-2 in the previous section will clearly show that
there is no direct way of getting to a lattice structure or the negation from information
of the type in conjunction with knowledge base about concepts represented in suitable way
unless the context is very special. 

As opposed to a lattice order in a \emph{rough order}, I use a parthood relation that is
reflexive and antisymmetric. It may be non transitive. The justification for using such a
relation can be traced to various situations in which restrictive criteria operating on
inclusion of attributes happen. In many cases, these may be dealt with using fuzzy
methodologies. Contexts using the so-called rough-fuzzy-rough approximations and
extensions thereof \cite{SPD} can be dealt with in a purely rough way through such
relations. The unary operations used in the definitions of the structures are intended as
approximation operators. More than two approximation operators are common in cover-based
\textsf{RST} \cite{ZW3}, dynamic spaces \cite{PP4}, Esoteric \textsf{RST} \cite{AM24},
multiple approximation spaces \cite{AQ1}, in dialectical rough set theory \cite{AM105} and
elsewhere. The requirement of equal number of upper and lower approximation operators is
actually a triviality and studies with non matching number of operators can be traced to
considerations on rough bottom and top equalities \cite{PN6}. Concrete examples are
provided after definitions.

The intended concept of a \emph{rough set} in a rough $Y$-system is as a collection of
some sense definite elements of the form $\{a_1, a_2, \ldots a_{n},\,b_{1}, b_{2},\ldots
b_r \}$ subject to $a_{i}$s being 'part of' of some of the $b_{j}$s.

Both \textsf{RYS+} and \textsf{RYS} can be seen as the generalization of the algebra
formed on the power set of the approximation space in classical \textsf{RST}. $\pc xy$ can
be read as 'x is a part of y' and is intended to generalise inclusion in the classical
case. The elements in $S$ may be seen as the collection of approximable and exact objects
- this interpretation is compatible with $S$ being a set. The description operator of FOPL
$\iota$ is used in the sense: $\iota(x) \Phi(x)$ means 'the $x$ such that $\Phi(x)$'. It
helps in improving expression and given this the metalogical $','$ can be understood as
$\wedge$ (which is part of the language). The description operator actually extends FOPL
by including more of the metalanguage and from the meaning point of view is different,
though most logicians will see nothing different. For details, the reader may refer to
\cite{WH}.

For those understanding '$,$' as being part of the metalanguage, statements of the form
$a+b = \iota(x)\Phi(x)$ can be safely read as $a+b = z $ if and only if $\Phi(z)$. It is
of course admissible to read '$,$' as being in the metalanguage with $\iota$ being part of
the language - the resulting expression would be more readable. 

\begin{definition}
A \emph{Rough Y System} (\textsf{RYS+}) will be a tuple of the form  \[\left\langle
{S},\,W,\,\pc ,\,(l_{i})_{1}^{n},\,(u_{i})_{1}^{n},\,+,\,\cdot,\,\sim,\,1 \right\rangle \]
satisfying all of the following ($\pc$ is intended as a binary relation on $S$ and
$W\,\subset\,S$, $n$ being a finite positive integer. $\iota$ is the description operator
of FOPL: $\iota(x) \Phi(x)$ means 'the $x$ such that $\Phi(x)$ '. $W$ is actually
superfluous and can be omitted):
\begin{enumerate}
\item {$(\forall x) \pc xx$ ; $(\forall x,\,y)(\pc xy,\,\pc yx\,\longrightarrow\,x=y)$,}
\item {For each $i,\,j$, $l_{i}$, $u_{j}$ are surjective functions $:S\,\longmapsto\,W$, }
\item {For each $i$, $(\forall{x,\,y})(\pc xy\,\longrightarrow\,\pc (l_{i}x)(l_{i}y),\,\pc
(u_{i}x)(u_{i}y))$,}
\item {For each $i$, $(\forall{x})\,\pc (l_{i}x)x,\,\pc (x)(u_{i}x))$,}
\item {For each $i$, $(\forall{x})(\pc
(u_{i}x)(l_{i}x)\,\longrightarrow\,x=l_{i}x=u_{i}x)$.}
\end{enumerate}

The operations $+,\,\cdot$ and the derived operations $\oc,\, \pp,\,\uc,\,\ox,\,\po $ will
be assumed to be defined uniquely as follows:
\begin{description}
\item [Overlap:]{$\oc xy\,\mathrm{iff} \,(\exists z)\,\pc zx\,\wedge\,\pc zy$,}
\item [Underlap:]{$\uc xy\,\mathrm{iff} \,(\exists z)\,\pc xz\,\wedge\,\pc yz$,}
\item [Proper Part:]{$\pp xy\, \mathrm{iff}\, \pc xy\wedge\neg \pc yx$,}
\item [Overcross:]{ $\ox xy \,\mathrm{iff}\,\oc xy\wedge\neg \pc xy$,}
\item [Proper Overlap:]{$\po xy \,\mathrm{iff}\,\ox xy \,\wedge\,\ox yx $,}
\item[Sum:]{$x+y=\iota z (\forall w)(\oc wz\, \leftrightarrow\,(\oc wx \vee \oc wy))$, }
\item[Product:]{$x \cdot y=\iota z (\forall w)(\pc wz\,\leftrightarrow\,(\pc wx \wedge \pc
wy))$, }
\item[Difference:]{$x - y=\iota z (\forall w)(\pc wz\,\leftrightarrow\,(\pc wx \wedge \neg
\oc wy))$, }
\item[Associativity:]{It will be assumed that $+,\,\cdot$ are associative operations and
so the corresponding operations on finite sets will be denoted by $\oplus,\,\odot$
respectively.}
\end{description}
\end{definition}

\begin{flushleft}
\textbf{Remark:} 
\end{flushleft}
$W$ can be dropped from the above definition and it can be required that the range of the
operations $u_{i}, l_{j}$ are all equal for all values of $i, j$.   

\begin{definition}
In the above definition, if we relax the surjectivity of $l_i ,u_i $, require partiality
of the operations $+$ and $\cdot$, weak associativity instead of associativity and weak
commutativity instead of commutativity, then the structure will be called a \emph{General
Rough Y-System} (\textsf{RYS}). In formal terms, 
\begin{description}
\item[Sum1:]{$x + y=\iota z (\forall w)(\oc wz\, \leftrightarrow\,(\oc wx \vee \oc wy))$
if defined}
\item[Product1:]{$x \cdot y=\iota z (\forall w)(\pc wz\,\leftrightarrow\,(\pc wx \wedge
\pc wy))$ if defined}
\item[wAssociativity]{$x\oplus (y\oplus z)\,\stackrel{\omega *}{=}(x\oplus y)\oplus z$ and
similarly for product. '$\stackrel{\omega *}{=}$' essentially means if either side is
defined, then the other is and the two terms are equal.}
\item[wCommutativity]{$x\oplus y\,\stackrel{\omega *}{=}\,y\oplus x $;  $x\cdot
y\,\stackrel{\omega *}{=}\,y\cdot x$}
\end{description}
\end{definition}

Both \textsf{RYS} and a \textsf{RYS+} are intended to capture a minimal common fragment of
different RSTs. The former is more efficient due to its generality. Note that the parthood
relation $\pc$, taken as a general form of rough inclusion (in a suitable semantic
domain), is not necessarily transitive. Transitivity of $\pc$ is a sign of fair choice of
attributes (at that level of classification), but non transitivity may result by the
addition of attributes and so the property by itself says little about the quality of
classification. 

\emph{The meaning of the equality symbol in the definition of \textsf{RYS} depends on the
application domain. It may become necessary to add additional stronger equalities to the
language or as a predicate in some situations. In this way, cases where any of conditions
$1, 3, 4, 5$ appear to be violated can be accommodated. All weaker equalities are expected
to be definable in terms of other equalities}.  

For example, using the variable precision \textsf{RST} procedures \cite{ZW,RR}, it is
possible to produce lower approximations that are not included in a given set and upper
approximations of a set that may not include the set. In \cite{AM24}, methods for
transforming the VPRS case are demonstrated. But nothing really needs to be done for
accommodating the VPRS case - the axioms can be assumed. The predicate $\pc$ would become
complicated as a consequence, though we can have $(\forall x, y)( x\,\subseteq\, y
\longrightarrow \pc xy )$. A stronger equality should be added to the language if
required.       

Vague predicates may be involved in the generation of \textsf{RYS} and \textsf{RYS+}.
Suppose crowds assembling at different places are to be comparatively studied in relation
to a database of information on 'relevant' people and others through audiovisual
information collected by surveillance cameras. Typically automatic surveillance equipment
will not be able to identify much, but information about various subsets of particular
crowds and even 'specific people' can be collected through surveillance cameras.
Processing information (so called off-line processing) using the whole database with
classical rough methods will not work because of scalability issues and may be just
irrelevant. Suppose that data about different gatherings have been collected at different
times. The collection of the observed subsets of these crowds can be taken as $S$. The
operations $l_{i}, u_{i}$ can originate on the basis of the capabilities of the
surveillance equipment like cameras. If one camera can see better in infra-red light,
another may see better in daylight, cameras do not have uniform abilities to move and
finally placement of the individual camera in question will affect its sight. Definite
abstract collections of people may be also taken as approximations of subsets of the
crowds based on information from the database and the set of these may form the $W$ of
\textsf{RYS+} or this may be a \textsf{RYS}. These can be used to define predicates like
'vaguely similar' between subsets of the crowd. Because of difficulties in scalability of
the computation process of identification, the collections $\underline{S}$ of possible
subsets of the crowd should be considered with a non-transitive parthood relation-based on
a criteria derived from inclusion of 'relevant' people and others (possibly number of
people with some gross characteristics), instead of set inclusion. The latter would
naturally lead to aggravation of errors and so should not be used.  As of now automated
surveillance is imperfect at best and so the example is very real. \textsf{RYS+} can also
be used to model classical rough set theory and some others close to it, but not esoteric
\textsf{RST} \cite{AM24} as approximations of 'definite objects' may not necessarily be
the same 'definite objects'. \textsf{RYS} on the other hand can handle almost all types of
\textsf{RST}. 

In the above two definitions, the parthood relation is assumed to be reflexive and
antisymmetric, while the approximation operators are required to preserve parthood.
Further any of the lower approximations of an object is required to be part of it and the
object is required to be part of any of its upper approximations. The fifth condition is a
very weak form of transitivity. The Venn diagram way of picturing things will not work
with respect to the mereology, but some intuitive representations may be constructed by
modification. Two objects \emph{overlap} if there is a third object that is part of both
the objects. Two objects \emph{underlap} if both are part of a third object. In general
such an object may not exist, while in ZF all sets taken in pairs will underlap. However
if the considerations are restricted to a collection of sets not closed under union, then
underlap will fail to hold for at least one pair of sets. \emph{Overcross} is basically
the idea of a third object being part of two objects, with the first being not a part of
the second. In the above example a set of 'relevant people' may be part of two subsets of
the crowd (at different times), but the first crowd may contain other people with blue
coloured hair. So the first crowd is not part of the second. If the second crowd contains
other people with their hair adorned with roses while such people are not to be located in
the first crowd then the two crowds have \emph{proper overlap}. 

From the purely mereological point of view a \textsf{RYS+} is a very complicated object.
The sum, product and difference operations are assumed to be defined. They do not follow
in general from the conditions on $\pc$ in the above. But do so with the assumptions of
closed extensional mereology or in other words of the first five of the following axioms.
They can also follow from the sixth (Top) and part of the other axioms.
 
\begin{description}
\item [Transitivity]{$(\pc xy,\,\pc yz\,\longrightarrow\,\pc xz)$,}
\item [Supplementation]{$(\neg \pc xy\,\longrightarrow\,\exists z (\pc zx\,\wedge\,\neg
\po zy))$,}
\item [P5]{$\uc xy\, \rightarrow\,(\exists z)(\forall w)(\oc wz \leftrightarrow (\oc wz
\,\vee\, \oc wy))$,}
\item [P6]{$\oc xy\,\rightarrow\,(\exists z)(\forall w)(\pc wz \leftrightarrow (\pc wz
\,\wedge\, \pc wy))$,}
\item [P7]{$(\exists z)(\pc zx\,\wedge\,\neg\oc zy)\,\rightarrow\,(\exists z)(\forall
w)(\pc wz \leftrightarrow (\pc wx \,\wedge\,\neg \oc wy))$,}
\item [Top]{$(\exists z)(\forall x) \pc xz$.}
\end{description}

In classical RST, 'supplementation' does not hold, while the weaker version $(\neg \pc
xy\,\longrightarrow\,\exists z (\pc zx\,\wedge\,\neg \oc zy))$ is trivially satisfied due
to the existence of the empty object ($\emptyset$).  Proper selection of semantic domains
is essential for avoiding possible problems of ontological innocence \cite{MCE}, wherein
the 'sum' operation may result in non existent objects relative the domain. A similar
operation results in 'plural reference' in \cite{PLS,PL}, and related papers. The
Lesniewski ontology inspired approach originating in \cite{PS3} assumes the availability
of measures beforehand for the definition of the parthood predicate and is not always
compatible with and distinct from the present approach.  

\begin{flushleft}
\textbf{Examples of Non-Transitivity}:\\ 
\end{flushleft}

\begin{flushleft}
\textbf{Example-1}:\\ 
\end{flushleft}

In the classical \emph{handle-door-house} example, parthood is understood in terms of
attributes and a level of being part of. The latter being understood in terms of
attributes again. The example remains very suggestive in the context of applications of
\textsf{RST} and specifically in the context of a \textsf{RYS}. The basic structure of the
example has the form:

\begin{itemize}\renewcommand{\labelitemi}{$\bullet$}
\item {Handle is part of a Door,}
\item {Door is part of a House,}
\item {If 'part of' is understood in the sense of 'substantial part of' (defined in terms
of attributes), then the handle may not be part of the house.}
\end{itemize}

From the application point of view all the concepts of 'Handle', 'Door' and 'House' can be
expected to be defined by sets of relatively atomic sensor (for machines) or sense data.
Additionally a large set of approximation related data (based on not necessarily known
heuristics) can also be expected. But if we are dealing with sensor data, then it can be
reasonable to assume that the data is the result of some rough evolution. Finding one is
an inverse problem. 

\begin{flushleft}
\textbf{Example-2}:\\ 
\end{flushleft}

\begin{itemize}\renewcommand{\labelitemi}{$\bullet$}
\item {Let Information Set-A be the processed data from a grey scale version of a colour
image. }
\item {Let 'Information-B' be the processed data about distribution of similar colours (at
some level).}
\item {In this situation, when 'Information set A' is processed in the light of
'Information set B', then transitivity of the parthood relations about colour related
attributes can be expected to be violated.}
\end{itemize}

This example is based on the main context of \cite{SPD}, but is being viewed from a pure
rough perspective.  
$\square$

\begin{flushleft}
\textbf{Example-3}:\\ 
\end{flushleft}

In processing information from videos in off-line or real time mode, it can be sensible to
vary the partitions on the colour space (used in analytics) across key frames.   $\square$

\begin{definition}
In the above, two approximation operators $u_i$ and $l_i$ will be said to be
\emph{S5-dual} to each other if and only if \[(\forall A\subset
S)\,A^{u_{i}l_{i}}=A^{u_{i}} \;;\; A^{l_{i} u_{i}}=A^{l_{i}} .\] 
\end{definition}

Throughout this paper it will not be assumed that the operators $u_{i}$ are S5-dual or
dual to the operators $l_{i}$ in the classical sense in general. It is violated in the
study of the lower, upper and bitten upper approximation operators in a tolerance space
\cite{AM105} as \textsf{RYS}. There it will also be required to take the identity operator
or repeat the lower approximation operator as a lower approximation operator (as the
number of lower and upper approximation are required to match - a trivial requirement). 

In almost all applications, the collection of all granules $\mathcal{G}$ forms a subset of
the \textsf{RYS} $S$. But a more general setting can be of some interest especially in a
semi-set theoretical setting. This aspect will be considered separately.  

\begin{definition}
When elements of $\mathcal{G}$ are all elements of $S$, it makes sense to identify these
elements with the help of a unary predicate $\gamma$ defined via, $\gamma x$ if and only
if $x\in \mathcal{G}$. A \textsf{RYS} or a \textsf{RYS+} enhanced with such a unary
predicate will be referred to as a \emph{Inner} \textsf{RYS} (or \YS for short) or a
\emph{Inner} \textsf{RYS+} (\YSP for short) respectively. \YS will be written as ordered
pairs of the form $(S,\,\gamma )$ to make the connection with $\gamma$ clearer.
$(S,\,\gamma)$ should be treated as an abbreviation for the algebraic system (partial or
total) \[\left\langle {S},\,\pc
,\,\gamma,\,(l_{i})_{1}^{n},\,(u_{i})_{1}^{n},\,+,\,\cdot,\,\sim,\,1 \right\rangle. \]   
 
\end{definition}

Some important classes of properties possibly satisfiable by granules fall under the broad
categories of \emph{representability}, \emph{crispness}, \emph{stability},
\emph{mereological atomicity} and \emph{underlap}.  If the actual representations are
taken into account then the most involved axioms will fall under the category of
representability. Otherwise the possible defining properties of a set of granules in a
\textsf{RYS} include the following ($t_{i},\,s_{i}$ are term functions formed with
$+,\,\cdot,\,\sim$, while $p,\,r$ are finite positive integers. $\mathbf{\forall}i$,
$\mathbf{\exists}i$ are meta level abbreviations.) Not all of these properties have been
considered in \cite{AM99}:

\begin{description}
\item [Representability, RA] {$\mathbf{\forall}i$, $(\forall x)(\exists y_{1},\ldots
y_{r}\in \mathcal{G})$ $y_{1} + y_{2} + \ldots + y_{r}=x^{l_{i}}$ and $(\forall x)(\exists
y_{1},\,\ldots\,y_{p}\in \mathcal{G})\,y_{1} + y_{2} + \ldots + y_{p}=x^{u_{i}}$,}
\item [Weak RA, WRA] {$\mathbf{\forall}i$, $(\forall x \exists y_{1},\ldots y_{r}\in
\mathcal{G})$ $t_{i}(y_{1},\,y_{2}, \ldots \,y_{r})=x^{l_{i}}$\\ and $(\forall x)(\exists
y_{1},\,\ldots\,y_{r}\in \mathcal{G})\,t_{i}(y_{1},\,y_{2}, \ldots \,y_{p}) = x^{u_{i}}$,}
\item [Sub RA]{$\mathbf{\exists}i$, $(\forall x)(\exists y_{1},\ldots y_{r}\in
\mathcal{G})$ $y_{1} + y_{2} + \ldots + y_{r}=x^{l_{i}}$\\ and $(\forall x)(\exists
y_{1},\,\ldots\,y_{p}\in \mathcal{G})\,y_{1} + y_{2} + \ldots + y_{p}=x^{u_{i}}$,}
\item [Sub TRA, STRA] {$\mathbf{\forall}i$, $(\forall x \exists y_{1},\ldots y_{r}\in
\mathcal{G})$ $t_{i}(y_{1},\,y_{2}, \ldots \,y_{r})=x^{l_{i}}$\\ and $(\forall x)(\exists
y_{1},\,\ldots\,y_{r}\in \mathcal{G})\,t_{i}(y_{1},\,y_{2}, \ldots \,y_{p}) = x^{u_{i}}$,}
\item [Lower RA, LRA] {$\mathbf{\forall}i$, $(\forall x)(\exists y_{1},\ldots y_{r}\in
\mathcal{G})$ $y_{1} + y_{2} + \ldots + y_{r}=x^{l_{i}}$,}
\item [Upper RA, URA] {$\mathbf{\forall}i$, $(\forall x)(\exists y_{1},\,\ldots\,y_{p}\in
\mathcal{G})\,y_{1} + y_{2} + \ldots + y_{p}=x^{u_{i}}$,}
\item [Lower SRA, LSRA] {$\mathbf{\exists}i$, $(\forall x)(\exists y_{1},\ldots y_{r}\in
\mathcal{G})$ $y_{1} + y_{2} + \ldots + y_{r}=x^{l_{i}}$,}
\item [Upper SRA, USRA] {$\mathbf{\exists}i$, $(\forall x)(\exists
y_{1},\,\ldots\,y_{p}\in \mathcal{G})\,y_{1} + y_{2} + \ldots + y_{p}=x^{u_{i}}$,}

\item [Absolute Crispness, ACG] {For each $i$, $(\forall y\in \mathcal{G})\,y^{l_{i}} =
y^{u_{i}} = y$,}
\item [Sub Crispness, SCG] {$\mathbf{\exists}i$, $(\forall y\in \mathcal{G}) y^{l_{i}} =
y^{u_{i}} = y$ (In \cite{AM99}, this was termed 'weak crispness'),}
\item [Crispness Variants] {LACG, UACG, LSCG, USCG will be defined as for
representability,}

\item [Mereological Atomicity,MER] {$\mathbf{\forall}i$, $(\forall y\in
\mathcal{G})(\forall x \in S)(\pc xy,\,x^{l_{i}} = x^{u_{i}} = x \longrightarrow x = y)$,}
\item [Sub MER,SMER] {$\mathbf{\exists}i$, $(\forall y\in \mathcal{G})(\forall x \in
S)(\pc xy,\,x^{l_{i}} = x^{u_{i}} = x \longrightarrow x = y)$ (In \cite{AM99}, this was
termed 'weak MER'),}
\item [Inward MER, IMER]{\[(\forall y\in \mathcal{G})(\forall x \in S)(\pc
xy,\,\bigwedge_{i}(x^{l_{i}} = x^{u_{i}} = x) \longrightarrow x = y),\]}
\item [Lower MER, LMER] {$\mathbf{\forall}i$, $(\forall y\in \mathcal{G})(\forall x \in
S)(\pc xy,\,x^{l_{i}} = x \longrightarrow x = y)$,}
\item [Inward LMER, ILMER] { $(\forall y\in \mathcal{G})(\forall x \in S)(\pc
xy,\,\bigwedge_{i} (x^{l_{i}} = x) \longrightarrow x = y)$,}
\item [MER Variants] {UMER, LSMER, USMER, IUMER will be defined as for representability,}

\item [Lower Stability, LS] {$\mathbf{\forall}i$, $(\forall y \in \mathcal{G})(\forall
{x\in S})\,(\pc yx\,\longrightarrow\,\pc (y)(x^{l_{i}}))$,}
\item [Upper Stability, US] {$\mathbf{\forall}i$, $(\forall y\in\mathcal{G})(\forall {x\in
S})\,(\oc yx \longrightarrow \pc(y)(x^{u_{i}}))$,}
\item [Stability, ST] {Shall be the same as the satisfaction of LS and US,}
\item [Sub LS, LSS] {$\mathbf{\exists}i$, $(\forall y \in \mathcal{G})(\forall {x\in
S})\,(\pc yx\,\longrightarrow\,\pc (y)(x^{l_{i}}))$ (In \cite{AM99}, this was termed
'LS'),}
\item [Sub US, USS] {$\mathbf{\exists}i$, $(\forall y\in\mathcal{G})(\forall {x\in
S})\,(\oc yx \longrightarrow \pc(y)(x^{u_{i}}))$ (In \cite{AM99}, this was termed 'US'),}
\item [Sub ST, SST] {Shall be the same as the satisfaction of LSS and USS,}

\item [No Overlap, NO] {$(\forall x ,\,y\in\mathcal{G}) \neg \po xy$,}

\item [Full Underlap, FU] { $\mathbf{\forall}i$, $(\forall x,\,y\in\mathcal{G})(\exists
z\in S ) \pp xz,\,\pp yz,\,z^{l_{i}} = z^{u_{i}} = z$,}
\item [Lower FU, LFU]{ $\mathbf{\forall}i$, $(\forall x,\,y\in\mathcal{G})(\exists z\in S
) \pp xz,\,\pp yz,\,z^{l_{i}} = z$,}
\item [Sub FU, SFU] { $\mathbf{\exists}i$, $(\forall x,\,y\in\mathcal{G})(\exists z\in S )
\pp xz,\,\pp yz,\,z^{l_{i}} = z^{u_{i}} = z$,}
\item [Sub LFU, LSFU]{ $\mathbf{\exists}i$, $(\forall x,\,y\in\mathcal{G})(\exists z\in S
) \pp xz,\,\pp yz,\,z^{l_{i}} = z$,}
\item [Unique Underlap, UU] {For at least one $i$,\\ $(\forall x, y\in\mathcal{G})(\pp xz,
\pp yz,\, z^{l_{i}}=z^{u_{i}}=z,\,\pp xb, \pp yb,\, b^{l_{i}}=b^{u_{i}}=b \longrightarrow
z=b)$,}
\item [Pre-similarity, PS] {$(\forall x ,\,y\in\mathcal{G})(\exists z \in \mathcal{G}) \pc
(x\cdot y)(z)$,}
\item [Lower Idempotence, LI] {$\mathbf{\forall}i$, $(\forall x \in\mathcal{G})
x^{l_{i}}\,=\,x^{l_{i}l_{i}}$,}
\item [Upper Idempotence, UI] {$\mathbf{\forall}i$, $(\forall x \in\mathcal{G})
x^{u_{i}}\,=\,x^{u_{i} u_{i}}$,}
\item [Idempotence, I] {$\mathbf{\forall}i$, $(\forall x \in\mathcal{G})
x^{u_{i}}\,=\,x^{u_{i} u_{i}},\,x^{l_{i}}\,=\,x^{l_{i} l_{i}}$.}
\end{description}

All of the above axioms can be written without any quantification over $\mathcal{G}$ in an
inner \textsf{RYS} or an inner \textsf{RYS+}. The letter 'I' for 'Inner' will be appended
to the axiom abbreviation in this situation. For example, I will rewrite \textbf{LS} as
\textbf{LSI}:\\
 
\[\mathbf{LSI}:\,\,\mathbf{\exists i}, (\forall {x, y\,\in S})\,(\gamma x ,\, \pc
yx\,\longrightarrow\,\pc (y)(x^{l_{i}})). \]
Further, statements of the form $(S,\gamma )\,\models \mathbf{RAI} \rightarrow (S,\gamma
)\,\models \mathbf{WRAI}$ ($\models$ being the model satisfaction relation in FOPL) will
be abbreviated by '\textbf{RAI} $\twoheadrightarrow$ \textbf{WRAI}'.  

\begin{proposition}
The following holds:
\begin{enumerate}
\item {\textbf{RAI} $\twoheadrightarrow$ \textbf{WRAI},}
\item {\textbf{ACGI} $\twoheadrightarrow$ \textbf{SCGI},}
\item {\textbf{MERI} $\twoheadrightarrow$ \textbf{SMERI},}
\item {\textbf{MERI} $\twoheadrightarrow$ \textbf{IMERI},}
\item {\textbf{FUI} $\twoheadrightarrow$ \textbf{LUI}.}
\end{enumerate}
\end{proposition}

The axioms \textsf{RA, WRA} are actually very inadequate for capturing representability in
the present author's opinion. Ideally the predicate relating the set in question to the
granules should be representable in a nice way. A hierarchy of axioms have been developed
for a good classification by the present author and will appear separately in a more
complete study. But I will not digress to this issue in the present paper.  

In any \textsf{RST}, at least some of these axioms can be expected to hold for a given
choice of granules. In the following sections various theorems are proved on the state of
affairs.

\subsection{Concepts Of Discernibility}

In 1-neighbourhood systems, cover-based \textsf{RST}, relation-based \textsf{RST} and more
generally in a \textsf{RYS} various types of indiscernibility relations can be defined. In
most cases, indiscernibility relations that are definable by way of conditions using term
functions involving approximation operators are of interest. Some examples of such
conditions, in a \textsf{RYS} of the form specified in the third section, are:

\begin{enumerate} \renewcommand\labelenumi{\theenumi}
  \renewcommand{\theenumi}{(\roman{enumi})}
\item {$x\approx_{i} y $ if and only if $x^{l_{i}} = y^{l_{i}}$ and $x^{u_{i}} =
y^{u_{i}}$ for a specific $i$,}
\item {$x\approx_{a}y $ if and only if $x^{l_{i}} = y^{l_{i}}$ and $x^{u_{i}} = y^{u_{i}}$
for any $i$,}
\item {$x\approx_{b}y $ if and only if $x^{l_{i}} = y^{l_{i}}$ and $x^{u_{i}} = y^{u_{i}}$
for all $i$,}
\item {$a\approx_{c}y $ if and only if $(\forall g\in \mathcal{G})(\pc
gx^{\alpha}\leftrightarrow \pc gy^{\alpha})$ with $\alpha\in\{l_{i},\,u_{i}\}$ for a
specific $i$,}
\item {$a\approx_{e}y $ if and only if $(\forall g\in \mathcal{G})(\pc
gx^{\alpha}\leftrightarrow \pc gy^{\alpha})$ with $\alpha\in\{l_{i},\,u_{i}\}$ for a
specific $i$,}
\item {$a\approx_{f}y $ if and only if $(\forall g\in \mathcal{G})(\pc
gx^{\alpha}\leftrightarrow \pc gy^{\alpha})$ with $\alpha\in\{l_{i},\,u_{i}\}$ for any
$i$,}
\item {$a\approx_{h}y $ if and only if $(\forall g\in \mathcal{G})(\pc
gx^{\alpha}\leftrightarrow \pc gy^{\alpha})$ with $\alpha\in\{l_{i},\,u_{i}\}$ for all
specific $i$.}
\end{enumerate}

Note that the subscript of $\approx$ has been chosen arbitrarily and is used to
distinguish between the different generalised equalities. Weaker varieties of such
indiscernibility relations can also be of interest.

\subsection{Relative- and Multi-Granulation}

Concepts of relativised granulation have been studied in a few recent papers
\cite{QL,QLY}, under the name 'Multi-Granulation'. These are actually granulations in one
application context considered relative the granulation of another application context.
For example if two equivalences are used to generate approximations using their usual
respective granulations, then 'multi-granulations' have been used according to authors.
The relative part is not mentioned explicitly but that is the intended meaning. In our
perspective all these are seen as granulations. Multiple approximation spaces, for example
use 'multi-granulations'.    

The relation between the two contexts has not been transparently formulated in the
mentioned papers, but it can be seen that there is a way of transforming the present
application context into multiple instances of the other context in at least one
perspective. In general it does not happen that approximations in one perspective  are
representable in terms of the approximation in another perspective (see \cite{QLY2}) and
is certainly not a requirement in the definition of multi-granulation. Such results would
be relevant for the representation problem of granules mentioned earlier.

\section{Relation-Based Rough Set Theory}

\begin{theorem}
In classical \textsf{RST}, if $\mathcal{G}$ is the set of partitions, then all of
\textsf{RA, ACG, MER, AS, FU, NO, PS} hold. \textsf{UU} does not hold in general 
\end{theorem}
\begin{proof}
The granules are equivalence classes and \textsf{RA, NO, ACG, PS} follow from the
definitions of the approximations and properties of $\mathcal{G}$. \textsf{MER} holds
because both approximations are unions of classes and no crisp element can be properly
included in a single class. If a class overlaps with another subset of the universe, then
the upper approximation of the subset will certainly contain the class by the definition
of the latter.     
\end{proof}

In esoteric \textsf{RST} \cite{AM24}, partial equivalence relations (symmetric, transitive
and partially reflexive relations) are used to generate approximations instead of
equivalence relations. In the simplest case, the upper and lower approximations of a
subset $A$ of a partial approximation space $\left\langle S, \, R \right\rangle$ are
defined via ($[x] = \{y;\, (x,y)\in R\}$ being the pseudo-class generated by $x$)
\[A^{l}\,=\,\bigcup\{[x];\, [x]\subseteq A\};\,\,A^{u}\,=\,\bigcup\{[x];\, [x]\cap A \neq
\emptyset\}.\]

\begin{theorem}
In case of esoteric \textsf{RST} \cite{AM24}, with the collection of all pseudo-classes
being the granules, all of \textsf{RA, MER, NO, UU, US} hold, but \textsf{ACG} may not. 
\end{theorem}
\begin{proof}
\textsf{RA, NO} follow from the definition. It is possible that $[x]\subset [x]^{u}$, so
\textsf{ACG} may not hold. \textsf{US} holds as if a granule overlaps another subset, then
the upper approximation of the set would surely include the granule.  
\end{proof}

If we consider a reflexive relation $R$ on a set $S$ and define, $[x]\,=\,\{y\,:\,
(x,y)\in R\}$ -the set of x-relateds and define the lower and upper approximation of a
subset $A\subseteq S$ via 
\[A^{l}\,=\,\cup \{[x]\,:\, [x]\subseteq A,\, x\in A \}\,\, \mathrm{and}\]

\[A^{u}\,=\,\cup \{[x]\,:\, [x]\cap A\,\neq\, \emptyset\, x\in A \},\]
($A^{l}\subseteq A\subseteq A^{u}$ for a binary relation $R$ is equivalent to its
reflexivity \cite{YY9,SLV})
then we can obtain the following about granulations of the form
$\{[x]\,:\,x\in S\}$:

\begin{theorem}
\textsf{RA, LFU} holds, but none of \textsf{MER, ACG, LI, UI, NO, FU} holds in general.  
\end{theorem}

\begin{proof}
\textsf{RA} holds by definition, \textsf{LFU} holds as the lower approximation of the
union of two granules is the same as the union. It is easy to define an artificial counter
example to support the rest of the statement.
\end{proof}

Let $\left\langle S,(R_{i})_i\in K \right\rangle $ be a multiple approximation space
\cite{AQ1}, then the strong lower, weak lower, strong upper and weak upper approximations
of a set $X\,\subseteq\, S$ shall be defined as follows (modified terminology): 
\begin{enumerate}
 \item {$X^{ls}\,=\,\bigcap_{i} X^{li}$,}
 \item {$X^{us}\,=\,\bigcup_{i} X^{ui}$,}
 \item {$X^{lw}\,=\,\bigcup_{i} X^{li}$,}
 \item {$X^{uw}\,=\,\bigcap_{i} X^{ui}$.}
\end{enumerate}

\begin{theorem}
 In a multiple approximation of the above form, taking the set of granules to be the
collection of all equivalence classes of the $R_{i}$s, \textsf{LSRA, USRA, LSS, USS}
holds, but all variants of rest do not hold always.
\end{theorem}
\begin{proof}
$X^{lw}$, $X^{us}$ are obviously unions of equivalence classes. The LSS, USS part for
these two approximations respectively can also be checked directly. Counterexamples can be
found in \cite{AQ1}. But it is possible to check these by building on actual
possibilities. If there are only two distinct equivalences, then at least two classes of
the first must differ from two classes of the second. The $ls$ approximation of these
classes will strictly be a subset of the corresponding classes, so \textsf{CG} will
certainly fail for the $(ls, us)$ pair. Continuing the argument, it will follow that
\textsf{SCG, ACG} cannot hold in general. The argument can be extended to other
situations.    
\end{proof}

Since multiple approximations spaces are essentially equivalent to
special types of tolerance spaces equipped with the largest equivalence contained in the
tolerance, the above could as well have been included in the following subsection.

\subsection{Tolerance Spaces}

In \textsf{TAS} of the form $\left\langle S,\,T \right\rangle $, all of the following
types of granules with corresponding approximations have been used in the literature:
\begin{enumerate}
\item {The collection of all subsets of the form $[x]\,=\,\{y: (x,y)\in T\}$ will be
denoted by $\mathcal{T}$.}
\item {The collection of all blocks, the maximal subsets of $S$ contained in $T$, will be
denoted by $\mathcal{B}$. Blocks have been used as granules in \cite{CG98,AM3,AM69,SW3}
and others.}
\item {The collection of all finite intersections of blocks will be denoted by
$\mathcal{A}$.}
\item {The collection of all intersections of blocks will be denoted by
$\mathcal{A}_{\sigma}$ \cite{SW3}.}
\item {Arbitrary collections of granules with choice functions operating on them
\cite{AM99}.}
\item {The collection of all sets formed by intersection of sets in $\mathcal{T}$ will be
denoted by $\mathcal{TI}$.}
\end{enumerate}

For convenience $H_{0}= \emptyset,\,H_{n+1} = S$ will be assumed whenever the collection
of granules $\mathcal{G}$ is finite and $\mathcal{G}\,=\,\{H_{1},\,\ldots\,H_{n}\}$.

In a \textsf{TAS} $\left\langle S,\,T \right\rangle $, for a given collection of granules
$\mathcal{G}$ definable approximations of a set $A\,\subseteq\,S$ include: 
\begin{enumerate}\renewcommand\labelenumi{\theenumi}
  \renewcommand{\theenumi}{(\roman{enumi})}
\item {$A^{l\mathcal{G}}\,=\,\bigcup\{H:\, H\subseteq A,\, H\in \mathcal{G} \}$,}
\item {$A^{u\mathcal{G}}\,=\,\bigcup\{H:\, H\cap A\,\neq \emptyset,\, H\in \mathcal{G}
\}$,}
\item {$A^{l2\mathcal{G}}\,=\,\bigcup\{\cap_{i \in I}\,H_{i}^{c}:\,\cap_{i\in I} H_{i}^{c}
\subseteq A,\, H\in \mathcal{G}\, I \subseteq \mathbf{N}(n+1) \}$,}
\item {$A^{u1\mathcal{G}}\,=\,\bigcap\{\cup_{i\in I}{H_{i}}\,:\,A\subseteq\, \cup_{i\in I}
H_{i},\,I \subseteq \mathbf{N}(n+1) \}$,}
\item {$A^{u2\mathcal{G}}\,=\,\bigcap\{{H_{i}}^{c}\,:\,A\subseteq\, H_{i}^{c},\,I \in
\{0,1,\ldots , n\} \}$.}
\end{enumerate}

But not all approximations fit it into these schemas in an obvious way. These include:

\begin{enumerate}\renewcommand\labelenumi{\theenumi}
\renewcommand{\theenumi}{(\roman{enumi})}
\item {$A^{l+} = \{y:\, \exists x (x, y)\in T,\,[x]\subseteq A \}$ \cite{CG98},}  
\item {$A^{u+} = \{x\,;\,(\forall
{y})\,((x,\,y)\,\in\,T\,\longrightarrow\,[y]_{T}\,\cap\,A\,\neq\,\emptyset)\} $,}
\item {Generalised bitten upper approximation : $A^{ubg} =A^{ug}\setminus A^{clg}$ - this
is a direct generalisation of the bitten approximation in \cite{AM105,SW}.}
\end{enumerate}

\begin{theorem}
In the \textsf{TAS} context, with $\mathcal{T}$ being the set of granules and restricting
to the approximations $l\mathcal{T}, u\mathcal{T}$, all of \textsf{RA, MER, ST} and
weakenings thereof hold, but others do not hold in general.  
\end{theorem}
\begin{proof}
\textsf{RA} follows from definition. For \textsf{MER}, if $A\subseteq [x]$ and
$A^{l\mathcal{T}}=A^{u\mathcal{T}}=A$, then as $[x]\cap A\neq \emptyset$, so $[x]\subseteq
A^{u\mathcal{T}} = A$. So $A=[x]$. Crispness fails because it is possible that
$[x]\cap[y]\neq \emptyset$ for distinct $x, y$.   
\end{proof}

\begin{theorem}
If $\left\langle \underline{S},\, T\right\rangle$ is a tolerance approximation space with
granules $\mathcal{T}$ and the approximations $l\mathcal{T},\,l+,\,u\mathcal{T},\,u+$,
then \textsf{RA, NO, ACG } do not hold, but \textsf{SRA, SMER, SST, IMER, MER, US} holds
\end{theorem}
\begin{proof}
\textsf{RA} does not hold due to $l+, u+$, \textsf{ACG} fails by the previous theorem.
'Sub' properties have been verified in the previous theorem, while the rest can be checked
directly.  
\end{proof}

\begin{theorem}
In Bitten \textsf{RST} \cite{AM105,SW}, (taking $\mathcal{G}$ to be the set of T-relateds
and restricting ourselves to the lower, upper and bitten upper approximations alone),
\textsf{SRA} does not hold for the bitten upper approximation if '$+,\,\cdot$' are
interpreted as set union and intersection respectively. \textsf{MER, NO} do not hold in
general, but \textsf{IMER, SCG, LS, LU, SRA } hold. 
\end{theorem}
\begin{proof}
The proof can be found in \cite{AM105}. If unions and intersections alone are used to
construct terms operations, then the bited upper approximation of a set of the form
$[x]_{T}$ ($x$ being an element of the tolerance approximation space) may not be
representable as it is the difference of the upper approximation and the lower
approximation of the complement of $[x]_{T}$. But if $\sim$, for set complements is also
permitted, then there are no problems with the representability in the \textsf{WRA} sense.
 
\end{proof}

In \cite{SW3}, a semantics of tolerance approximation spaces for the following theorem
context are considered, but the properties of granules are not mentioned.
\begin{theorem}
Taking $\mathcal{G}$ to be $\mathcal{A}_{\sigma}$ and restricting ourselves to lower and
bitten upper approximations alone \textsf{RA, ACG, NO} do not hold while \textsf{LRA, MER,
LACG, LMER, UMER, ST } do hold. 
\end{theorem}

\begin{proof}
If $H$ is a granule, then it is an intersection of blocks. It can be deduced that the
lower approximation of $H$ coincides with itself, while the bitten upper approximation is
the union of all blocks including $H$. LRA is obvious, but URA need not hold due to the
bitten operation. If a definite set is included in a granule, then
it has to be a block that intersects no other block and so the granule should coincide
with it. So \textsf{MER} holds.
\end{proof}

\section{Cover-Based Rough Set Theory}

The notation for approximations in cover-based \textsf{RST}, is not particularly
minimalistic. This is rectified for easy comprehension below. I follow superscript
notation strictly. '$l, u$' stand for lower and upper approximations and anything else
following those signify a type. 
If $X\subseteq  S$, then let
\begin{enumerate} \renewcommand\labelenumi{\theenumi}
  \renewcommand{\theenumi}{(\roman{enumi})}
\item {$X^{u1+}\,=\,X^{l1}\cup\bigcup\{\md(x):\, x\in X \} $,}
\item {$X^{u2+}\,=\,\bigcup \{K:\,K\in\ms , K \cap X\neq\emptyset\} $,}
\item {$X^{u3+}\,=\,\bigcup \{\md(x):\, x\in X\} $,}
\item {$X^{u4+}\,=\,X^{l1}\cup \{K:\, K\cap (X\setminus X^{l1})\neq \emptyset\} $,}
\item {$X^{u5+}\,=\,X^{l1}\cup \bigcup \{nbd(x):\,x\in X\setminus X^{l1}\} $,}
\item {$X^{u6+}\,=\,\{x:\,nbd(x)\cap X\neq \emptyset\} $,}
\item {$X^{l6+}\,=\,\{x:\, nbd(x)\subseteq X \} $.}
\end{enumerate}
 
The approximation operators $u1+, \ldots, u5+$ (corresponding to first, ..., fifth
approximation operators used in \cite{ZW3} and references therein) are considered with the
lower approximation operator $l1$ in general. Some references for cover-based \textsf{RST}
include \cite{WZ,PO,CP4,TYL3,YY9,ZW3,ZHU3,IM,AM960}. The relation between cover-based
\textsf{RST} and relation-based \textsf{RST} are considered in \cite{AM960,ZW3}. For a
cover to correspond to a tolerance, it is necessary and sufficient that the cover be
normal - a more general version of this result can be found in \cite{CZ}. When such
reductions are possible, then good semantics in Meta-R perspective are possible.  The main
results of \cite{AM960} provide a more complicated correspondence between covers and sets
equipped with multiple relations or a relation with additional operations. The full scope
of the results are still under investigation. So, in general, cover-based \textsf{RST} is
more general than relation-based \textsf{RST}. From the point of view of expression,
equivalent statements formulated in the former would be simpler than in the latter.    

It can be shown that:

\begin{proposition} 
In the above context,
\begin{itemize}\renewcommand{\labelitemi}{$\bullet$}
 \item {$\md(x)$ is invariant under removal of reducible elements,}
 \item {$(nbd(x))^{l6+}\,=\,nbd(x)$,}
 \item {$nbd(x)\,\subseteq\,(nbd(x))^{6+}$.}
\end{itemize}
\end{proposition}

The following pairs of approximation operators have also been considered in the literature
(the notation of \cite{CP4} has been streamlined; $lp1,\,lm1$ corresponds to
$\underline{P}_{1},\,\underline{C}_{1}$ respectively and so on).

\begin{enumerate} \renewcommand\labelenumi{\theenumi}
  \renewcommand{\theenumi}{(\roman{enumi})}
\item{$X^{lp1}\,=\,\{x:\,Fr(x)\subseteq X\} $,}  
\item{$X^{up1}\,=\,\bigcup\{K: K\in\mathcal{K},\,K\cap X\neq \emptyset \}$,}
\item{$X^{lp2}\,=\,\bigcup\{Fr(x);\, Fr(x)\subseteq X\} $,}  
\item{$X^{up2}\,=\,\{z:\,(\forall y)(z\in Fr(y)\rightarrow Fr(y)\cap X\neq \emptyset)\}$,}
\item{$X^{lp3}\,=\,X^{l1} $,}  
\item{$X^{up3}\,=\,\{y:\,\forall {K}\in \mathcal{K}(y\in K\rightarrow K\cap X\neq
\emptyset)\}$,}
\item{$X^{lp4}, X^{up4}$ are the same as the classical approximations with respect to
$\pi(\mathcal{K})$ - the partition generated by the cover $\mathcal{K}$.}  
\item{$X^{lm1}\,=\,X^{l1}\,=\,X^{lp3} $,}  
\item{$X^{um1}\,=\,X^{u2}$,}
\item{$X^{lm2}\,=\,X^{l6+}$,}  
\item{$X^{um2}\,=\,X^{u6+}$,}
\item{$X^{lm3}\,=\,\{x;\,(\exists u) u\in nbd(x), nbd(u)\subseteq X\} $,}  
\item{$X^{um3}\,=\,\{x;\, (\forall u)(u\in nbd(x)\rightarrow nbd(u)\cap X\neq
\emptyset)\}$,}
\item{$X^{lm4}\,=\,\{x; (\forall x)(x\in nbd(u)\rightarrow nbd(u)\subseteq X)\} $,}  
\item{$X^{um4}\,=\,X^{u6+}\,=\,X^{um2}$,}
\item{$X^{lm5}\,=\,\{x;\,(\forall u)(x\in nbd(u)\rightarrow u\in X)\}$,}
\item{$X^{um5}\,=\,\bigcup\{nbd(x);\,x\in X\}$.}
\end{enumerate}

\begin{flushleft}
\textbf{Example-1}:\\ 
\end{flushleft}

Let $S\,=\,\{a, b, c, e, f, g, h, i,j\}$,
\[\mathcal{K}\,=\,\{K_{1},\,K_{2},\,K_{3},\,K_{4},\,K_{5},\,K_{6},K_{7},\,K_{8},\,K_{9}\},
\]
\[K_{1}=\{a,\,b\},\, K_{2} = \{a,\,c,\,e\},\,K_{3} = \{b,\,f\},\,K_{4} = \{j \},\,K_{5} =
\{f,\,g,\,h\},\]
\[K_{6} = \{ i\},\,K_{7} = \{f,\,g,\,j,\,a \},\,K_{8} = \{f,\,g \},\,K_{9} = \{a,\,j \}.\]

The following table lists some of the other popular granules:

\begin{center}
\begin{tabularx}{330pt}{|c|X|X|X|}
\hline\hline
Element: $x$ & $Fr(x)$ &  $Md(x)$ & $nbd(x)$\\
\hline
$a$ & $S\setminus\{h,\,i \} $ & $\{K_{1}, K_{2}, K_{3} \} $ & $\{a\} $\\
\hline
$b$ & $\{a,\,b,\,f \} $ & $\{K_{3} \} $ & $\{b\} $\\
\hline
$c$ & $\{a,\,c,\,e \} $ & $\{K_{2} \} $ & $\{a,\,c,\,e\} $\\
\hline
$e$ & $\{a,\,c,\,e\} $ & $\{K_{2} \} $ & $\{a,\,c,\,e\} $\\
\hline
$f$ & $S\setminus\{c,\,e,\,i\} $ & $\{K_{3}, K_{8}\} $ & $\{f\} $\\
\hline
$g$ & $\{a,\,f,\,g,\,h,\,j \} $ & $\{K_{8} \} $ & $\{f,\,g\} $\\
\hline
$h$ & $\{f,\,g,\,h \} $ & $\{K_{5} \} $ & $\{f,\,g,\,h\} $\\
\hline
$i$ & $\{i\} $ & $\{K_{6} \} $ & $\{i\} $\\
\hline
$j$ & $\{a,\,f,\,g,\,j\} $ & $\{K_{4}\} $ & $\{j\} $\\
\hline
& & & \\
\hline
\end{tabularx}
\end{center}

\hrulefill

The best approach to granulation in these types of approximations is to associate an
initial set of granules and then refine them subsequently. Usually a refinement can be
expected to be generable from the initial set through relatively simple set theoretic
operations. In more abstract situations, the main problem would be of representation and
the results in these situations would be the basis of possible abstract representation
theorems. The theorems proved below throw light on the fine structure of granularity in
the cover-based situations:

\begin{theorem}
In AUAI \textsf{RST} \cite{IM,AM24}, with the collection of granules being $\mathcal{K}$
and the approximation operators being ($l_{1},\, l_{2},\, u_{1}$ and $u_{2}$),
\textsf{WRA, LS, SCG, LU, IMER} holds, but \textsf{ACG, RA, SRA, MER} do not hold in
general. 
\end{theorem}
\begin{proof}
\textsf{WRA} holds if the complement, union and intersection operations are used in
the construction of terms in the \textsf{WRA} condition. \textsf{ACG} does not hold as the
elements of $\mathcal{K}$ need not be crisp with respect to $l2$. Crispness holds with
respect to $l1, u1$, so \textsf{SCG} holds. \textsf{MER} need not hold as it is violated
when a granule is properly included in another. IMER holds as the pathology negates the
premise. It can also be checked by a direct argument.
\end{proof}

From the definition of the approximations in AUAI and context, it should be clear that all
pairings of lower and upper approximations are sensible generalisations of the classical
context. In the following two of the four are considered. These pairs do not behave well
with respect to duality, but are more similar with respect to representation in terms of
granularity. 

\begin{theorem}
In AUAI \textsf{RST}, with the collection of granules being $\mathcal{K}$ and the
approximation operators being  $(l_{1},\, u_{1})$, \textsf{WRA, ACG, ST, LU} holds, but
\textsf{MER, NO, FU, RA} do not hold in general. 
\end{theorem}

\begin{theorem}
In AUAI \textsf{RST}, with the collection of granules being $\mathcal{K}$ and the two
approximation operators being $(l_{2},\, u_{2})$,  \textsf{WRA, ST} holds, but
\textsf{ACG, MER, RA, NO} do not hold in general. 
\end{theorem}
\begin{proof}
If $K\in \mathcal{K}$, $K\subseteq X$ ($X$ being any subset of $S$) and $y\in K^{l2}$,
then $y$ must be in at least one intersection of the sets of the form $S\setminus K_{i}$
(for $i\in I_{0}$, say) and so it should be in each of these $S\setminus K_{i}\subseteq
K\subseteq X$. This will ensure $y\in X^{l2}$. So lower stability holds. Upper stability
can be checked in a similar way.  
\end{proof}

\begin{theorem}
When the approximations are $(lp1,\, up1)$ and with the collection of granules being
$\{Fr(x)\}$, all of \textsf{MER, URA, UMER} hold, but \textsf{ACG, NO, LS} do not hold
necessarily   
\end{theorem}
\begin{proof}
For an arbitrary subset $X$, $X\subseteq Fr(x)$ for some $x\in S$ and $X^{lp1} = X^{up1} =
X$ would mean that $Fr(x) =X$ as $Fr(x)\subseteq X^{up1}$ would be essential. So UMER and
the weaker condition MER holds. \textsf{URA} is obvious from the definitions.
\end{proof}

\begin{theorem}
When the approximations are $(lp2,\, up2)$ and with the collection of granules being
$\{Fr(x)\}$, all of \textsf{MER, LMER,  RA, LCG} hold, but \textsf{ACG, NO, LS} do not
hold
necessarily   
\end{theorem}
\begin{proof}
From the definition, it is clear that \textsf{RA, LCG} hold. If for an arbitrary subset
$X$, $X\subseteq Fr(x)$ for some $x\in S$ and $X^{lp2} = X = X^{up2}$, then $X$ is a union
of some granules of the form $Fr(y)$. If $x\in X$, then it is obvious that $X = Fr(x)$.

If $x\in Fr(x)\setminus X$ and $F(x)$ is an element of the underlying cover $\mathcal{S}$,
then again it would follow that $X=Fr(x)$. Finally if $x\in Fr(x)\setminus X$ and $Fr(x)$
is a union of elements of the cover intersecting $X$, then we would have a contradiction.
So MER follows.

If for an arbitrary subset
$X$, $X\subseteq Fr(x)$ for some $x\in S$ and $X^{lp2} = X$ and $x\in Fr(x)\setminus X$,
then we will have a contradiction. So \textsf{LMER} holds.
\end{proof}

\begin{theorem}
When the approximations are $(lp3,\, up3)$ and with the collection of granules being
$\mathcal{K}$, all of \textsf{MER, RA, ST, LCG, LU} hold, but \textsf{ACG, NO} do not hold
necessarily.   
\end{theorem}
\begin{proof}
Both the lower and upper approximations of any subset of $S$ is eventually a union of
elements $\mathcal{K}$, so \textsf{RA} holds. Other properties follow from the
definitions. Counter examples are easy.  
\end{proof}

\begin{theorem}
When the approximations are $(lp4,\, up4)$ and with the collection of granules being
$\pi(\mathcal{K})$, all of \textsf{RA, ACG, MER, AS, FU, NO, PS} hold.
\end{theorem}
\begin{proof}
With the given choice of granules, this is like the classical case.  
\end{proof}

\begin{theorem}
When the approximations are $(lp4,\, up4)$ and with the collection of granules being
$\mathcal{K}$, all of \textsf{WRA, ACG, AS} hold, while the rest may not.
\end{theorem}
\begin{proof}
\textsf{WRA} holds because elements of $\pi(\mathcal{K})$ can be represented set
theoretically in terms of elements of $\mathcal{K}$. Lower and upper approximations of
elements of $\mathcal{K}$ are simply unions of partitions of the elements induced by
$\pi(\mathcal{K})$.
\end{proof}

\begin{theorem}
When the approximations are $(lm1,\, um1)$ and with the collection of granules being
$\mathcal{K}$, all of \textsf{WRA, LS, LCG} hold, but \textsf{RA, ST, LMER} do not hold
necessarily. For \textsf{WRA}, complementation is necessary.   
\end{theorem}
\begin{proof}
If $\mathcal{K}$ has an element properly included in another, then \textsf{LMER} will
fail. If complementation is also permitted, then WRA will hold. Obviously \textsf{RA} does
not hold. Note the contrast with the pair $(lp3, up3)$ with the same granulation.  
\end{proof}

\begin{theorem}
When the approximations are $(lm2,\, um2)$ and with the collection of granules being
$\mathcal{N}$, all of \textsf{LCG, LRA, ST, MER} holds, but \textsf{RA, ACG, LMER, NO} do
not hold necessarily.   
\end{theorem}
\begin{proof}
If $y\in nbd(x)$ for any two $x,\,y\in S$, then $nbd(y)\subseteq nbd(x)$ and it is
possible that $x\notin nbd(y)$, but it is necessary that $x\in nbd(x)$. So $(nbd(x))^{lm2}
= nbd(x)$, but $(nbd(x))^{um2}$ need not equal $nbd(x)$. \textsf{LRA} will hold as the
lower approximation will be a union of neighbourhoods, but this need happen in case of the
upper approximation. \textsf{NO} is obviously false. The upper approximation of a
neighbourhood can be a larger neighbourhood of a different point. So \textsf{ACG} will not
hold in general.
\textsf{MER} can be easily checked. 
\end{proof}

\begin{theorem}
When the approximations are $(l6+,\, u6+)$ and with the collection of granules being
$\mathcal{N}$, all of \textsf{LCG, LRA, ST, MER} holds, but \textsf{RA, ACG, LMER, NO} do
not hold necessarily.   
\end{theorem}
\begin{proof}
Same as above.  
\end{proof}

\begin{theorem}
When the approximations are $(l1,\, u1+)$ and with the collection of granules being
$\mathcal{K}$, all of \textsf{ACG, RA, FU, LS} holds, but \textsf{MER, LMER, NO} do not
hold necessarily.   
\end{theorem}
\begin{proof}
\textsf{RA} holds as all approximations are unions of granules. For any granule $K$,
$K^{l1}= K$ and so $K^{u1+} = K^{l} = K$. So \textsf{ACG} holds. If for two granules $A,
B$, $A\subset B$, then $A^{l1}=A^{u1+}=A$, but $A\neq B$. So $MER$, $LMER$ cannot hold in
general.    
\end{proof}

\begin{theorem}
When the approximations are $(l1,\, u2+)$ and with the collection of granules being
$\mathcal{K}$, all of \textsf{ACG, RA, FU, ST} holds, but \textsf{MER, LMER, NO} do not
hold necessarily.   
\end{theorem}
\begin{proof}
\textsf{RA} holds as all approximations are unions of granules. For any granule $K$,
$K^{l1}= K$ and so $K^{u1+} = K^{l} = K$. So \textsf{ACG} holds. If for two granules $A,
B$, $A\subset B$, then $A^{l1}=A^{u1+}=A$, but $A\neq B$. So \textsf{MER, LMER} cannot
hold in general. If a granule $K$ is included in a subset $X$ of $S$, then it will be
included in the latter's lower approximation. If $K$ intersects another subset, then the
upper approximation of the set will include $K$. So \textsf{ST} holds.    
\end{proof}

\begin{theorem}
When the approximations are $(l1,\, u3+)$ and with the collection of granules being
$\mathcal{K}$, all of \textsf{ACG, RA, FU, LS} holds, but \textsf{MER, LMER, NO} do not
hold necessarily.   
\end{theorem}
\begin{proof}
\textsf{RA} holds as all approximations are unions of granules. For any granule $K$,
$K^{l1}= K$ and so $K^{u1+} = K^{l} = K$. So \textsf{ACG} holds. If for two granules $A,
B$, $A\subset B$, then $A^{l1}=A^{u3+}=A$, but $A\neq B$. So \textsf{MER, LMER} cannot
hold in general. The union of any two granules is a definite element, so \textsf{FU}
holds. 
\end{proof}

\begin{theorem}
When the approximations are $(l1,\, u4+)$ and with the collection of granules being
$\mathcal{K}$, all of \textsf{ACG, RA, FU, LS} holds, but \textsf{MER, LMER, NO} do not
hold necessarily   
\end{theorem}
\begin{proof}
\textsf{RA} holds as all approximations are unions of granules. For any granule $K$,
$K^{l1}= K$ and so $K^{u1+} = K^{l} = K$. So \textsf{ACG} holds. If for two granules $A,
B$, $A\subset B$, then $A^{l1}=A^{u3+}=A$, but $A\neq B$. So \textsf{MER, LMER} cannot
hold in general. The union of any two granules is a definite element, so \textsf{FU}
holds. 
\end{proof}

\begin{theorem}
When the approximations are $(l1,\, u5+)$ and with the collection of granules being
$\mathcal{K}$, all of \textsf{ACG, RA, FU, LS} holds, but \textsf{MER, LMER, NO} do not
hold necessarily.   
\end{theorem}
\begin{proof}
\textsf{RA} holds as all approximations are unions of granules. For any granule $K$,
$K^{l1}= K$ and so $K^{u1+} = K^{l} = K$. So \textsf{ACG} holds. If for two granules $A,
B$, $A\subset B$, then $A^{l1}=A^{u3+}=A$, but $A\neq B$. So \textsf{MER, LMER} cannot
hold in general. The union of any two granules is a definite element, so \textsf{FU}
holds. 
\end{proof}

Apparently the three axioms \textsf{WRA, LS, LU} hold in most of the known theories and
with most choices of granules. This was the main motivation for the following definition
of admissibility of a set to be regarded as a set of granules.    

\begin{definition}
A subset $\mathcal{G}$ of $S$ in a \textsf{RYS} will be said to be an \emph{admissible set
of granules} provided the properties \textsf{WRA}, \textsf{LS} and \textsf{LU} are
satisfied by it. 
\end{definition}

In cover-based \textsf{RST}s, different approximations are defined with the help of a
determinate collection of subsets. These subsets satisfy the properties \textsf{WRA},
\textsf{LS} and \textsf{FU} and are therefore admissible granules.  But they do not in
general have many of the nicer properties of granules in relation to the approximations.
However, at a later stage it may be possible to refine these and construct a better set of
granules (see \cite{AM960}, for example) for the same approximations. Similar process of
refinement can be used in other types of RSTs as well. For these reasons, the former will
be referred to as \emph{initial granules} and the latter as relatively \emph{refined
granules}. It may happen that more closely related approximations may as well be formed by
such process.

\subsection{Classification of Rough Set Theory}

From the point of view of logic or rough reasoning associated, \textsf{RST} can be
classified according to:
\begin{enumerate}
\item {General context and definition of approximations.}
\item {Modal logic perspective from Meta-C (see \cite{BK3}).}
\item {Frechet Space perspective from Meta-C \cite{TYL3}.}
\item {Global properties of approximations viewed as operators at Meta-C (see for example
\cite{YY9}).}
\item {Rough equality based semantic perspective at Meta-R (see for example \cite{AM24}).}
\item {Granularity Based Perspective (this paper).}
\item {Algebraic perspectives at Meta-C.}
\item {Algebraic perspectives at Meta-R.}
\item {Others.}
\end{enumerate}

In general the meta-C classification is not coherent with meta-R features. The problems
are most severe in handling quotients induced by rough equalities. In the algebraic
perspective, the operations at meta-C level are not usually preserved by quotients. 

For algebraic approaches at Meta-C, the classification goes exactly as far as duality (for
formulation of logic ) permits. Modal approaches can mostly be accommodated within the
algebraic. But the gross classification into  
relation-based, cover-based and more abstract forms of \textsf{RST} remains coherent with
desired duality. Of course, the easier cases that fit into the scheme of \cite{TYL3} can
be explored in more ways from the Meta-C perspective. The common algebraic approaches to
modal logics further justifies such a classification as:

\begin{itemize}\renewcommand{\labelitemi}{$\bullet$}
\item {The representation problem of algebraic semantics corresponds most naturally to the
corresponding problems in duality theory of partially or lattice-ordered algebras or
partial algebras. Some of the basic duality results for relation-based \textsf{RST} are
considered in \cite{JJ}. \cite{OR} is a brief survey of applications of topology free
duality. For studies proceeding from a sequent calculus or modal axiomatic system point of
view, the classification also corresponds to the difficulty of the algebraization problem
in any of the senses \cite{PPM2,BK3,OE3,CH}.}
\item {The duality mentioned in the first point often needs additional topology in
explicit terms. The actual connection of the operator approach in \textsf{RST} with point
set topology is : Start from a collection with topological or pre-topological operators on
it, form possibly incompatible quotients and study a collection of some of these types of
objects with respect to a new topology (or generalisations thereof) on it. This again
supports the importance of the classification or the mathematical uniqueness of the
structures being studied.}
\end{itemize}

The present axiomatic approach to granules does provide a level of classification at
Meta-C. But the way in which approximations are generated by granules across the different
cases is not uniform and so comparisons will lack the depth to classify Meta-R dynamics,
though the situation is definitely better than in the other mentioned approaches. One way
out can be through the representation problem. It is precisely for this reason that the
classification based on the present axiomatic approach is not stressed.  

For those who do not see the point of the contamination problem, the axiomatic theory
developed provides a supportive classification for general \textsf{RST}.

\section{Dialectical Counting, Measures}

To count a collection of objects in the usual sense it is necessary that they be distinct
and well defined. So a collection of well defined distinct objects and indiscernible
objects can be counted in the usual sense from a higher meta level of perception. Relative
this meta level, the collection must appear as a set. In the counting procedures
developed, the use of this meta level is minimal and certainly far lesser than in other
approaches.     
It is dialectical as two different interpretations are used simultaneously to complete the
process. These two interpretations cannot be merged as they have contradictory information
about relative discernibility. Though the classical interpretation is at a relatively
higher meta level, it is still being used in the counting process and the formulated
counting is not completely independent of the classical interpretation. 

A strictly formal approach to these aspects will be part of a forthcoming paper.    

Counting of a set of objects of an approximation space and that of its power set are very
different as they have very different kinds of indiscernibility inherent in them. The
latter possess a complete evolution for all of the indiscernibility present in it while
the former does not. Counting of elements of a \textsf{RYS} is essentially a
generalisation of the latter. In general any lower level structure like an approximation
space corresponding to a 1-neighbourhood system \cite{YY9} or a cover need not exist in
any unique sense. The axiomatic theory of granules developed in the previous sections
provides a formal exactification of these aspects.  

Let $S$ be a \textsf{RYS}, with $R$ being a derived binary relation (interpreted as a weak
indiscernibility relation) on it. As the other structure on $S$ will not be explicitly
used in the following, it suffices to take $S$ to be an at most countable set of elements
in ZF, so that it can be written as a sequence of the form:
$\{x_{1},\,x_{2},\,\ldots,\,x_{k},\,\ldots ,\,\}$. Taking $(a,b) \in R$  to mean '$a$ is
weakly indiscernible from $b$' concepts of \emph{primitive counting} regulated by
different types of meta level assumptions are defined below. The adjective
\emph{primitive} is intended to express the minimal use of granularity and related axioms.

\subsection*{Indiscernible Predecessor Based Primitive Counting (IPC)}

In this form of 'counting', the relation with the immediate preceding step of counting
matters crucially.
\begin{enumerate}
\item {Assign $f(x_{1})\, = \,1_{1} =s^{0}(1_{1})$.}
\item {If $f(x_{i})=s^{r}(1_{j})$ and $(x_{i},x_{i+1}) \in  R$, then assign
$f(x_{i+1})=1_{j+1}$.}
\item {If $f(x_{i})=s^{r}(1_{j})$ and $(x_{i},x_{i+1}) \notin  R$, then assign
$f(x_{i+1})=s^{r+1}(1_{j})$.}
\end{enumerate}

The 2-type of the expression $s^{r+1}(1_{j})$ will be $j$. Successors will be denoted by
the natural numbers indexed by 2-types. 

\subsection*{History Based Primitive Counting (HPC)}
In HPC, the relation with all preceding steps of counting will be taken into account. 
\begin{enumerate}
\item {Assign $f(x_{1}) =  1_{1} = s^{0}(1_{1})$.}
\item {If $f(x_{i}) = s^{r}(1_{j})$ and $(x_{i},x_{i+1}) \in  R$, then assign $f(x_{i+1})
= 1_{j+1}$. }
\item {If $f(x_{i}) = s^{r}(1_{j})$ and $\bigwedge_{k < i+1} (x_{k},x_{i+1}) \notin R$,
then assign $f(x_{i+1}) = s^{r+1}(1_{j})$. }
\end{enumerate}

In any form of counting in this section, if $f(x)=\alpha$, then $\tau (\alpha)$ will
denote the least element that is related to $x$, while $\epsilon (\alpha)$ will be the
greatest element preceding $\alpha$, which is related to $\alpha$.

\subsection*{History Based Perceptive Partial Counting (HPPC)}
In HPPC, the valuation set shall be $N\,\cup\,\{*\}=N^{*}$ with $*$ being an abbreviation
for 'undefined'.
\begin{enumerate}
\item {Assign $f(x_{1}) = 1 = s^{0}(1)$.}
\item {If $(x_{i},x_{i+1}) \in  R$, then assign $f(x_{i+1}) = *$ . }
\item {If $Max_{k< i} f(x_{k}) = s^{r}(1)$ and $\bigwedge_{k < i} (x_{k},x_{i}) \notin 
R$, then assign  $f(x_{i}) = s^{r+1}(1)$. }
\end{enumerate}
Clearly, HPPC depends on the order in which the elements are counted.

\subsection*{Indiscernible Predecessor Based Partial Counting (IPPC)}
This form of counting is similar to HPPC, but differs in using the IPC methodology.
\begin{enumerate}
\item {Assign $f(x_{1}) = 1 = s^{0}(1)$.}
\item {If $(x_{i},x_{i+1}) \in  R$, then assign $f(x_{i+1}) = *$ .}
\item {If $Max_{k< i} f(x_{k})=s^{r}(1)$ and $(x_{i-1},x_{i})\notin R$, then assign
$f(x_{i})=s^{r+1}(1)$. }
\end{enumerate}

\begin{definition}
A generalised approximation space $\left\langle S,\,R\right\rangle $ will be said to be
\emph{IPP Countable} (resp HPP Countable) if and only if there exists a total order on
$S$, under which all the elements can be assigned a unique natural number under the IPPC
(resp HPPC) procedure . The ratio of number of such orders to the total number of total
orders will be said to be the \emph{IPPC Index} (resp HPPC Index).
\end{definition}

These types of countability are related to measures, applicability of variable precision
rough methods and combinatorial properties. The counting procedures as such are strongly
influenced by the meaning of associated contexts and are not easy to compare. In classical
\textsf{RST}, HPC can be used to effectively formulate a semantics, while the IPC method
yields a
weaker version of the semantics in general. This is virtually proved below:

\begin{theorem}
In the HPC procedure applied to classical \textsf{RST}, a set of the form
$\{l_{p},\,1_{q},\,\ldots 1_{t}\}$ is a granule if and only if it is the greatest set of
the form with every element being related to every other and for any element   $\alpha$ in
it $\tau(\alpha)=f^{\dashv}(l_{p})$ and $\epsilon(\alpha)=f^{\dashv}(1_{t})$ (for at least
two element sets). Singleton granules should be of the form $\{l_{p}\}$ with this element
being unrelated to all other elements in $S$.  
\end{theorem}

\begin{theorem}
All of the following are decidable in a finite number of steps in a finite \textsf{AS} by
the application of the HPC counting method :
\begin{enumerate}
\item {Whether a given subset $A$ is a granule or otherwise.}
\item {Whether a given subset $A$ is the lower approximation of a set $B$.}
\item {Whether a given subset $A$ is the upper approximation of a set $B$.}
\end{enumerate}
\end{theorem}

\begin{proof}
Granules in the classical \textsf{RST} context are equivalence classes and elements of
such
classes relate to other elements within the class alone.
\end{proof}

\begin{theorem}
The above two theorems do not hold in the IPC context. The criteria of the first theorem
defines a new type of granules in totally ordered approximation spaces and relative this
the second theorem may not hold. 
\end{theorem}

\begin{proof}
By the IPC way of counting, equivalence classes can be split up into many parts and the
numeric value may be quite deceptive. Looking at the variation of the 2-types, the
criteria of the first theorem above will define parts of granules, that will yield lower
and upper approximations distinct from the classical ones respectively.    
\end{proof}

Extension of these counting processes to TAS for the abstract extraction of information
about granules, becomes too involved even for the simplest kind of granules. However if
the IPC way of counting is combined with specific parthood orders, then identification of
granules can be quite easy. The other way is to consider entire collections of countings
according to the IPC scheme. 

\begin{theorem}
Among the collection of all IPC counts of a RYS $S$, any one with maximum number of
$1_{i}$s occurring in succession determines all granules. If $S$ is the power set of an
approximation space, then all the definite elements can also be identified from the
distribution of sequences of successive $1_{i}$s.   
\end{theorem}

\begin{proof}
\begin{itemize}\renewcommand{\labelitemi}{$\bullet$}
\item {Form the collection $\mathcal{I}(S)$ of all IPC counts of $S$.}
\item {For $\alpha, \beta\in \mathcal{I}(S)$, let $\alpha \preceq \beta$ if and only if
$\beta$ has longer strings of $1_{i}$ s towards the left and more number of strings of
$1_{i}$s of length greater than 1, than $\alpha$. }
\item {The maximal elements in this order suffice for this theorem.}
\end{itemize}
\end{proof}

\subsection*{Example}

Let $S=\{f, b, c, a, k, i, n, h, e, l, g, m \}$ and let $R,\,Q$ be the reflexive and
transitive closure of the relation \[\{(a,b),\,(b, c),\,(e, f),\, (i, k), (l, m),\, (m,
n),\,(g, h) \}\] and \[\{(a,b),\,(e,f),\,(i,k),\,(l,m),\,(m,n)\}\] respectively. Then
$\left\langle S,\,R\right\rangle $ and $\left\langle S,\,Q\right\rangle $ are
approximation spaces. This set can be counted (relative $R$) in the presented order as
follows:

\begin{description}
\item [IPC]{$\{1_{1}, 2_{1}, 1_{2},\,1_{3},\, 2_{3}, 1_{4}, 2_{4}, 3_{4}, 1_{5}, 2_{5},
1_{6}, 2_{6} \}$,} 
\item [HPC]{$\{1_{1}, 2_{1}, 1_{2}, 1_{3},\, 2_{3}, 1_{4}, 1_{5}, 2_{5}, 1_{6},\, 1_{7},
1_{8}, 1_{9} \}$,}
\item [HPPC]{$\{1_{1}, 2_{1}, *, *, 3_{1}, *, 4_{1}, 5_{1}, *, *, *, * \}$,}
\item [HPC]{Relative $Q$: $\{1_{1}, 2_{1}, 3_{1}, 1_{2}, 2_{2}, 1_{3}, 2_{3}, 3_{3},
1_{4},1_{5}, 2_{5}, 1_{6}\} $.}
\end{description}

Now $S|Q=\{\{a, b\}, \{c\},\{e,f\}, \{i, k\}, \{l, m, n \},\{g\},\{h\}\}$ and the positive
region of $Q$
relative $R$ is \[POS_{R}(Q)=\bigcup_{X\in S|Q} X^{l}=\{e,f, l, m, n\} = \{f, n, e, l,
m\}.\] 
The induced HPC counts of this set are respectively $\{1_{1}, 2_{5}, 1_{6}, 1_{7},
1_{9}\}$ and $\{1_{1}, 2_{3}, 1_{4}, 1_{5}, 1_{6}\}$.

\section{Generalized Measures}

According to Pawlak's approach \cite{ZPB} to theories of knowledge from a classical rough
perspective, if $\underline{S}$ is a set of attributes and $R$ an indiscernibility
relation on it, then sets of the form $A^{l}$ and $A^{u}$ represent clear and definite
concepts. If $Q$ is another stronger equivalence ($Q\,\subseteq\,R$) on $\underline{S}$,
then the state of the knowledge encoded by $\left\langle \underline{S},\,Q \right\rangle $
is a \emph{refinement} of that of $S=\left\langle\underline{S},\,R\right\rangle $. The
$R$-positive region of Q is defined to be \[POS_{R}(Q)\,=\,\bigcup_{X\in S|Q}
X^{l_{R}}\,;\;\;\,X^{l_{R}}\,=\,\bigcup\{[y]_{R};\,[y]_{R}\subseteq X\}.\]The degree of
dependence of knowledge $Q$ on $R$ $\delta(Q,R$ is defined by \[\delta(Q, R) =
\dfrac{Card(POS_{R}(Q))}{Card (S)}.\]

\begin{definition}
The \emph{granular dependence degree of knowledge $Q$ on $R$}, $gk(Q,R)$ will be the tuple
$(k_{1}, \ldots , k_{r} )$, with the $k_{i}$'s being the ratio of the number of granules
of type $i$ included in $POS_{R}(Q)$ to $card (S)$. 
\end{definition}

Note that the order on $S$, induces a natural order on the granules (classes) generated by
$R,\,Q$ respectively. This vector cannot be extracted from a single HPC count in general
(the example in the last section should be suggestive). However if the granulation is
taken into account, then much more information (apart from the measure) can be extracted. 

\begin{proposition}
 If $gk(Q, R) = (k_{1}, k_{2},\ldots, k_{r} )$, then $\delta(Q, R) = \sum k_{i}$.
\end{proposition}

The concepts of consistency degrees of multiple models of knowledge introduced in
\cite{CHP} can also be improved by a similar strategy:

If $\delta(Q,R) = a$ and $\delta(R,Q) = b$, then the consistency degree of Q and R,
$Cons(Q,R)$ is defined by  \[Cons(Q,R) = \dfrac{a+b+nab}{n+2}, \] where $n$ is the
consistency constant. With larger values of $n$, the value of the consistency degree
becomes smaller.

\begin{definition}
If $gk(Q,R) = (k_{1}, k_{2},\ldots, k_{r} )$ and $gk(R, Q) = (l_{1}, l_{2},\ldots, l_{p}
)$ then
the \emph{granular consistency degree} $gCons(Q, R)$ of the knowledge represented by $Q,
R$, will be   
\[gCons(Q,R) = ({k_1^*},\ldots, {k_r^*}, {l_1^*},\ldots, {l_p^*},{n k_{1}^{*}  l_{1}},
\dots {n k_{r}^{*} l_{p}}),\] where $k_i^* = \frac{k_{i}}{n+2}$ for $i = 1, \ldots, r$ and
$l_j^*=\frac{l_j}{n+2}$ for $j=1, \ldots, p$.
\end{definition}

The knowledge interpretation can be extended in a natural way to other general RST
(including \textsf{TAS}) and also to choice inclusive rough semantics \cite{AM99}.
Construction of similar measures is however work in progress. With respect to the counting
procedures defined, these two general measures are relatively constructive provided
granules can be extracted. This is possible in many of the cases and not always. They can
be replaced by a technique of defining atomic sub-measures based on counts and
subsequently combining them. 

Algebraic semantics for these and related measures have been developed in \cite{AM909} by
the present author (after the writing of this paper). The reader is referred to the
same for a fuller treatment of this section. 

\subsection{Rough Inclusion Functions}

Various rough inclusion and membership functions with related concepts of degrees are
known in the literature (see \cite{AG3} and references therein). If these are not
representable in terms of granules through term operations formed from the basic ones
\cite{AM99}, then they are not truly functions/degrees of the rough domain. To emphasize
this aspect, I will refer to such measures as being \emph{non-compliant for the rough
context}. I seek to replace such non-compliant measures by tuples satisfying the criteria.
Based on this heuristic, I would replace the rough inclusion function 
\[k(X, Y)=
\left\{
\begin{array}{ll}
\dfrac{\#(X\cap Y)}{\#(X)},  & \mathrm{if}\,\, X\neq \emptyset, \\
1, & \mathrm{else,}
\end{array}
\right.\] with 

\[k^{*}(X, Y)=
\left\{
\begin{array}{ll}
\left(y_{1},y_{2}, \ldots, y_{r}\right)   & \mathrm{if}\,\, X^{l}\neq \emptyset, \\
\left(\dfrac{1}{r},\ldots,\dfrac{1}{r}\right), & \mathrm{else}
\end{array}
\right.\]

where \[\dfrac{\#(G_{i})\cdot\chi_{i}(X\cap Y)}{\#(X^{l}) = y_{i}},\;\;i=1, 2, \ldots r.\]

Here it is assumed that $\{G_{1},\ldots , G_{r}\} = \mathcal{G}$ (the collection of
granules) and that the function $\chi_{i}$ is being defined via,
\[\chi_{i}(X)=
\left\{
\begin{array}{ll}
1, & \mathrm{if} G_{i}\subseteq X, \\
0, & \mathrm{else,}
\end{array}\right.\] 

Similarly,
\[k_{1}(X, Y)=
\left\{
\begin{array}{ll}
\dfrac{\#(Y)}{\#(X\cup Y)},  & \mathrm{if}\,\, X\cup Y \neq \emptyset, \\
1, & \mathrm{else,}
\end{array}\right.\]

can be replaced by 

\[k_{1}^{*}(X, Y)=
\left\{
\begin{array}{ll}
\left( h_1, h_{2}, \ldots,h_r\right)   & \mathrm{if}\, X^{l}\neq \emptyset,\\
\left(\frac{1}{r},\ldots,\frac{1}{r}\right), & \mathrm{else}
\end{array}\right.\]

where \[\dfrac{\#(G_{i})\cdot \chi_{i}(Y)}{\#((X\cup Y)^{l})} = h_{i}\,; i=1,2, ...,r,\]

and

\[k_{2}(X, Y)= \dfrac{\#(X^{c}\cup Y)}{\#(S)},\]

can be replaced by 

\[k_{2}^{*}(X, Y)= \left(q_{1}, q_{2},\ldots, q_{r} \right),\]
 
where \[\dfrac{\#(G_{1})\cdot\chi_{1}(X^{c} \cup Y)}{\#(S)} = q_{i},\;\;i=1,2,\ldots,r.\] 

This strategy can be extended to every other non-compliant inclusion function. Addition,
multiplication and their partial inverses for natural numbers can be properly generalised
to the new types of numbers with special regard to meaning of the operations on elements
of dissimilar type. This paves the way for the representation of these general measures in
the new number systems (given the granulation).

\section{On Representation of Counts}

In this section, ways of extending the representation of the different types of counts
considered in the previous section are touched upon. The basic aim is to endow these with
more algebraic structure through higher order constructions. The generalised counts
introduced have been given a representation that depend on the order of arrangement of
relatively discernible and indiscernible things from a Meta-C perspective. A
representation that accommodates all possible arrangements is of natural interest from
both concrete and abstract perspectives. The global structure that corresponds to $Z$ or
$N$ is not this and may be said to correspond to rather superfluous implicit
interpretations of counting exact objects by natural numbers. From a different perspective
a distinction is made between the names of objects and the 'generalised numbers'
associated in all this.

Consider the following set of statements:
\begin{description}
\item[A] {Set $S$ has finite cardinality $n$. }
\item[B] {Set $S$ can be counted in $n!$ ways.}
\item[C] {The set of counts $\mathcal{C}(S)$ of the set $S$ has cardinality $n!$.}
\item[E] {The set of counts $\mathcal{C}(S)$ of the set $S$ bijectively corresponds to the
permutation group $S_{n}$.}
\item[F] {The set of counts $\mathcal{C}(S)$ can be endowed with a group operation so that
it becomes isomorphic to $S_{n}$.}
\end{description}

Statements C, E and F are rarely explicitly stated or used in mathematical practice in the
stated form. To prove 'F', it suffices to fix an initial count. From the point of view of
information content it is obvious that $A\subset B\subset C \subset E \subset F$. It is
also possible to replace 'set' in the above statements by 'collection' and then from the
axiomatic point of statements C, E, and F will require stronger axioms (This aspect will
not be taken up in the present paper). I will show that partial algebraic variants of
statement F correspond to IPC. 

An important aspect that will not carry over under the generalisation is:
\begin{proposition}
Case F is fully determined by any of the two element generating sets (for finite $n$). 
\end{proposition}

\subsection{Representation Of IPC}

The group $S_{n}$ can be associated with all the usual counting of a collection of $n$
elements. The composition operation can be understood as the action of one counting on
another. This group can be associated with the RYS being counted and 'similarly counted
pairs' in the latter can be used to generate a partial algebra. Formally, the structure
will be deduced using knowledge of $S_{n}$ and then abstracted:  

\begin{itemize}\renewcommand{\labelitemi}{$\bullet$}
\item {Form the group $S_{n}$ with operation $\ast$ based on the interpretation of the
collection as a finite set of $n$ elements at a higher meta level,}
\item {Using the information about similar pairs, associate a second interpretation $s(x)$
based on the IPC procedure with each element $x\in S_{n}$,}
\item {Define $(x,y)\in \rho$ if and only if $s(x) = s(y)$.}
\end{itemize}

On the quotient $S_{n}|\rho$, let \[a\circ b\,=\,\left\lbrace  \begin{array}{ll}
 c & \mathrm{if}\,\;\{z:\,x\ast y=z,\, x\in a,\,y\in b \}|\rho \in S_{n}|\rho,\\
 \mathrm{\infty} & \mathrm{else}. \\
 \end{array} \right.\]

The following proposition follows from the form of the definition. 

\begin{proposition}
The partial operation $\circ$ is well defined. 
\end{proposition}

\begin{definition}
A partial algebra of the form $\left\langle S_{n}|\rho,\,\circ\right\rangle $ defined
above will be called a \emph{Concrete IPC-partial Algebra} (CIPCA) 
\end{definition}

A partially ordered set $\left\langle F, \leq \right\rangle $ is a \emph{lower
semi-linearly ordered set} if and only if:
\begin{itemize}\renewcommand{\labelitemi}{$\bullet$}
\item {For each $x$, $x\downarrow$ is linearly ordered,}
\item {$(\forall x, y)(\exists z) z\leq x \wedge z \leq y$. }
\end{itemize}

It is easy to see that IPC counts form lower semi-linearly ordered sets. Structures of
these types and properties of their automorphisms are all of interest and will be
considered in a separate paper.

\section{Semantics from Dialectical Counting}

An algorithmic method for deducing the algebraic semantics of classical \textsf{RST} using
the IPC method of counting is demonstrated in this section. Derivation of semantics for
other types of \textsf{RST} will be taken up in a separate paper.

Let $S = \left\langle\underline{S},\,R \right\rangle $ be an approximation space, then
$\left\langle\underline{\wp(S)},\,\cup,\,\cap, ^{l},\,^{u}, ^{c},\,0,\,1 \right\rangle$
can be regarded as a \textsf{RYS}. $0,\,1$, corresponding to the empty set and $S$, can be
regarded as distinguished elements or defined through additional axioms. The parthood
operation corresponds to set inclusion and is definable using $\cap,\,\cup$.
Indiscernibility of any $x,y \in \wp(S)$ will be assumed to be defined by the rough
equality; \[x\approx y \;\mathrm{iff}\; x^l = y^l\,\wedge\,  x^u = y^u.\]

If $\mathcal{I}(\wp(S))$ is the set of all IPC counts of the \textsf{RYS}, then by the
last theorem of section on dialectical counting it is clear that the granules and definite
elements can be identified. It remains to define the other operations on the identified
objects along the lines of \cite{BC1}. The corresponding logics in modal perspective can
be found in \cite{BCB}. So,   

\begin{theorem}
The rough algebra semantics of classical \textsf{RST} can be deduced from any one of the
maximal elements of $\mathcal{I}(\wp(S))$ in the $\preceq$ order.
\end{theorem}

\subsection{Rough Entanglement: Granular HPC}

By the term \emph{Entangled}, I mean to grasp the sharp increase in semantic content that
happens when counting processes are made to take aspects of the distribution of granules
into account. For the following version of counting, it is assumed that the set $S$ is
associated with a \textsf{RYS} $\mathcal{S}$, generated by it according to some process
and that a
collection of granules $\mathcal{G}$ is included in the \textsf{RYS} $\mathcal{S}$. It is
also
assumed that a total order on $S$ is available relative the lower level classical semantic
domain. These assumptions can be relaxed. The following definitions will also be used:
\begin{itemize}\renewcommand{\labelitemi}{$\bullet$}
\item {For $a, b\in S$ and $G\in \mathcal{G}$, $ind_{G} (a, b)$ if and only if $a,\,b \in
G$.}
\item {For $a, b\in S$ $disc(a,\,b)$ if and only if $(\exists G_{1}, G_{2}\in
\mathcal{G})$ $a\in G_{1}, b\in G_{2}, a\notin G_{2}, b\notin G_{1}$ (or the pair
guarantees $disc(a,\, b)$).}
\item {For $a, b\in S$ $pdisc(a,\,b)$ if and only if $(\exists G_{1},G_{2}\in
\mathcal{G})$ $a, b \in G_{1}, b\in G_{2},  a\notin G_{2}$.}
\end{itemize}

The granular part of the granular extension of HPC can be done as follows:
\begin{enumerate}
\item {Assume a total order on $\mathcal{G}$ and form $\mathcal{G}^{2}\,\setminus\,
\Delta$,}
\item {Form words with elements (alphabets) from $\mathcal{G}^{2}\,\setminus\, \Delta$ and
order them lexicographically,}
\item {A word that includes all instances of pairs that guarantee $disc(a,b)$ will be said
to be \emph{complete} for $(a, b)$, }
\item {A word will be said to be \emph{reduced} with respect to a pair $(a, b)$ if and
only if no letter $x$ in the word guarantees  $pdisc(a, b)$.}
\item {The existence of unique complete least reduced words (least with respect to the
lexicographic order, provided $\mathcal{G}$ is finite) for pairs of elements in $S$
permits simple valuation as a tuple and adds another dimension to HPC.}
\item {It suffices to adjoin the tuple corresponding to the encoded word between the
element and its predecessor in the counting procedure of HPC.}
\end{enumerate}

\begin{theorem}
Any two distinct elements $a,\,b \in S$ generate a least unique complete reduced word from
$\mathcal{G}$.   
\end{theorem}

\begin{proof}
I will need to use transfinite induction in the absence of the finiteness assumption.
Otherwise, forming the set of complete words, and then extracting the reduced words and
finally ordering the rest by the lexicographic graphic will yield the least element. A
contradiction argument can be used to check the uniqueness.
\end{proof}

\section{Connections Between Fuzzy and Rough Sets}

In this section, I will use granularity-related features to establish a new transformation
between fuzzy sets and classical rough sets. The transformation can be the basis for
translations or transformations between different types of logics associated with the
corresponding semantics. In the literature of the different attempts at connections
between fuzzy and rough sets, the non granular approach via the Brouwer Zadeh MV algebras
in \cite{CC} is somewhat related. The exact connections with the present derivations,
however, requires further investigation especially when rough membership functions in the
classical sense are not permitted into the discourse. The relation of this result to the
connections based on membership functions as developed in \cite{YYFR} and earlier papers
are mentioned at the end of this section. The result is very important in the light of the
contamination problem and the generalised number systems of \cite{AM1005}. It shows that
both concepts of rough measures and 'fuzzy sets as maps' are not needed to speak of rough
sets, fuzzy sets and possible connections between them. 

A non-controversial definition of fuzzy sets, with the purpose of removing the problems
with the 'membership function formalism', was proposed in \cite{RV}. I show that the
definition can be used to establish interesting links between fuzzy set theory and rough
sets. The connection is essentially of a mathematical nature. In my view, the results
should be read as \emph{in a certain perspective, the granularity of particular rough
contexts originate from fuzzy contexts and vice versa}. The existence of any such
perspective and its possible simplicity provides another classification of general rough
set theories.

\begin{definition}
A \emph{fuzzy subset} (or \emph{fuzzy set}) $ \mathbb{A}$ of a set $S$ is a collection of
subsets $\{A_{i}\}_{i\in [0,1]}$ satisfying the following conditions:
\begin{itemize}\renewcommand{\labelitemi}{$\bullet$}
\item {$A_{0} = S$,}  
\item {$(0\leq a\, <\, b\leq 1 \rightarrow A_{b}\subseteq A_{a})$,}
\item {$A_{b}\,=\,\bigcap_{0\leq a < b} A_{a}$.}
\end{itemize}
\end{definition}

A fuzzy membership function $\mu_{\mathbb{A}}: X\mapsto [0,1]$ is defined via
$\mu_{\mathbb{A}}(x) = \mathrm{Sup} \{a:\,a \in [0,1],\, x\in A_{a}\}$ for each $x$. The
\emph{core} of $\mathbb{A}$ is defined by $\mathrm{Core}(A) = \{x\in S: \,
\mu_{\mathbb{A}}(x)= 1\}$. $\mathbb{A}$ is \emph{normalized} if and only if it has
non-empty core. The \emph{support} of $\mathbb{A}$ is defined as the closure of $\{x \in
S;\,\mu_{\mathbb{A}}(x) > 0\}$. The \emph{height} of $\mathbb{A}$ is $H(\mathbb{A})=
\mathrm{Sup}\{\mu_{\mathbb{A}}(x);\,x\in S\}$.
The \emph{upper level set} is defined via $U(\mu, a)=\{x \in S:\, \mu_{\mathbb{A}}(x)\geq
a\}$. The class of all fuzzy subsets of $S$ will be denoted by $\mathcal{F}(S)$. The
standard practice is to refer to 'fuzzy subsets of a set' as simply a 'fuzzy set'.

\begin{proposition}
Every fuzzy subset $\mathbb{A}$ of a set $S$ is a granulation for $S$ which is a
descending chain with respect to inclusion and with its first element being $S$.
\end{proposition}

The cardinality of the indexing set and the second condition in the definition of fuzzy
sets is not a problem for use as granulations in \textsf{RST}, but almost all types of
upper approximations of any set will end up as $S$. From the results proved in the
previous sections it should also be clear that many of the nice properties of granulations
will not be satisfied modulo any kind of approximations. I will show that simple set
theoretic transformations can result in better granulations. Granulations of the type
described in the proposition will be called \emph{phi-granulations}. 

\begin{flushleft}
\textbf{Construction-1}: 
\end{flushleft}
\begin{enumerate}
\item {Let $P=\{0, p_{1}, \ldots p_{n-1}, 1\}$ be a finite set of rationals in the
interval [0,1] in increasing order.} 
\item {From $\mathbb{A}$ extract the collection $\mathbb{B}$ corresponding to the index
$P$.}
\item {Let $B_{0}\setminus B_{p_1} = C_{1}$, $B_{p_1}\setminus B_{p_2} = C_{2}$ and so on.
}
\item {Let $\mathcal{C} = \{C_{1}, C_{2}, \ldots,\, C_{n}\}$.}
\item {This construction can be extended to countable and uncountably infinite $P$ in a
natural way.}
\end{enumerate}

\begin{theorem}
The collection $\mathcal{C}$ formed in the fourth step of the above construction is a partition of $S$. The reverse transform is possible, provided $P$ has been selected in an appropriate way. 
\end{theorem}

It has been shown that fuzzy sets can be corresponded to classical rough sets in at least
one way and conversely by way of stipulating granules and selecting a suitable transform.
But a full semantic comprehension of these transforms cannot be done without imposing a
proper set of restrictions on admissibility of transformation and is context dependent.
The developed axiomatic theory makes these connections clearer.  

The result also means that rough membership functions are not necessary to establish a
semantics of fuzzy sets within the rough semantic domain as considered in \cite{YYFR}.
Further as noted in \cite{YYFR}, the semantics of fuzzy sets within rough sets is quite
restricted and form a special class. The core and support of a fuzzy set is realized as
lower and upper approximations. Here this need not happen, but it has been shown that any
fuzzy set defined as in the above is essentially equivalent to a granulation that can be
transformed into different granulations for \textsf{RST}. A more detailed analysis of the
connections will appear separately. 

\section{Operations on Low-Level Rough Naturals}

The operations that should be defined are those that have meaning and correspond to
semantic actions. An important aspect of the meaning relates to strings associated with
counts. In case of natural numbers, the number $2$ may be associated with two distinct
objects, but an IPC count will have a Meta-C sequence of objects which may be Meta-R
distinct or indiscernible from their predecessors and different kinds of operations can be
performed on these. Addition for example needs to be defined over strings and then
corresponded to new IPC counts.   

As mentioned earlier, two different types of structures relative the classical meta level
can be associated with possible concepts of 'rough naturals'. With respect to IPC for
example, the raw IPC count may be seen as a rough natural or the set of IPC counts of a
collection of things may be seen as a rough natural. I will refer to the former as the
\emph{low-level} concept of rough natural and the latter as a \emph{high-level} concept.
Operations on the former can be expected to influence the structure of the latter. In this
section operations on collections of low-level rough naturals will be considered. 

of representing other possible operations in a easy way):

If $x$ is an IPC count, then the string of relatively discernible and indiscernible
objects (abstract) associated with the count will be denoted by $\overline{x}$. When I
speak of 'the string corresponding to $x$', it can be understood as 
\begin{itemize}\renewcommand{\labelitemi}{$\bullet$}
\item {Any of the strings that have the same kind of distribution of discernible and
indiscernible objects with respect to their immediate predecessors, or}
\item {the abstract classes associated, or}
\item {a representative 'arbitrary' object (I am not going into the philosophical
viewpoints on the issue in this paper).}
\end{itemize}
Obviously this makes the appropriate concept of $\overline{x}$ for a method of counting to
depend on the method. Specifically, the IPC concept of $\overline{x}$ is not suitable for
getting nice structure in HPC contexts. The rough equality (indiscernibility or similarity
or whatever) relation will be denoted by $\approx$. The basic string operations will then
be defined as follows: 

If $I_{k}(x)$ is the Meta-R count of the Meta-C string formed by interchanging
$\overline{x}_{k}$ and $\overline{x}_{k+1}$ in $\overline{x}$.

\[\iota_{k}(x)=
\left\{
\begin{array}{ll}
x,  & \mathrm{if}\,\, (\overline{x}_{k},\overline{x}_{k+1})\in\approx, \\
I_{k}(x)  & \mathrm{otherwise.}
\end{array}
\right.\]

If $y$ is the Meta-R count of the result of the Meta-C removal of $\overline{x}_{k}$ from
$\overline{x}$, ( $x\neq 0$)

\[\varrho_{k}(x)=
\left\{
\begin{array}{ll}
y,  & \mathrm{if}\,\,\neg (\overline{x}_{k},\overline{x}_{k+1}) \in \approx\,
\mathrm{or}\,\neg (\overline{x}_{k-1},\overline{x}_{k}) \in \approx ,\\
x  & \mathrm{otherwise.}
\end{array}
\right.\] 

If $z$ is the IPC count of the result of insertion of the string $\overline{y}$ between
$\overline{x}_{k}$ and $\overline{x}_{k+1}$, then for non-zero $x$, $y$,

\[\eta_{k}(x,\,y)=
\left\{
\begin{array}{ll}
z , & \mathrm{if}\,\,(\overline{x}_{k},\overline{x}_{k+1}) \in \approx ,\\
x & \mathrm{otherwise.}
\end{array}
\right.\]

\begin{definition}
For any of the rough counts $x$, $\nu(x)$ will represent the length  of (cardinal
associated with) $\overline{x}$ in Meta-C.
\end{definition}

Note that $\eta_{k}(x,\,y)$ does not include $\oplus$, but is very similar. In fact the
two operations are not distinct from the point of view of interpretation at Meta-C. 

\begin{definition}
Basic arithmetical operations can be generalised to rough naturals as follows:
\begin{enumerate}\renewcommand\labelenumi{\theenumi}
  \renewcommand{\theenumi}{(\roman{enumi})} 
\item {$x\mathbf{\oplus} y$:  The count (at meta-R) of the 'object level string' (relative
meta-C) defined as the $\overline{x}$ items followed by the $\overline{y}$ items. }
\item {$x \mathbf{\odot} y$:  The count (at meta-R) of the 'object level string' defined
as $\nu (y)$ copies of $\overline{x}$ placed at each of the atomic $\overline{y}$ places.}
\item {$\mu(x, y)$:  Form $\nu (y)$ copies of $x$ at meta-C.}
\item {$x^{c}  $:   Meta-R Count of $\overline{x}$ in reverse order. }
\item {$x \otimes y$: The meta-R count of the string formed by the conditional
substitution of $\overline{x}$ at each of the $\overline{y}$ places subject to the rules: 
\begin{enumerate}
\item {Replace $\overline{y}_{1}$ with $\overline{x}$.}
\item {Replace $\overline{y}_{k+1}$ with $\overline{x}$ if
$(\overline{y}_{k},\overline{y}_{k+1})\notin\approx$.}
\item {If $(\overline{y}_{k},\overline{y}_{k+1})\in\approx$, then $\overline{y}_{k+1}$
should be dropped.}
\end{enumerate} 
}
\item {$x\ominus y$:  The meta-R count of the string formed by the deletion of
$\overline{y}$ from $\overline{x}$ from the right if defined. This definition can always
be improved through the $\varrho_{k}$ operation.}
\end{enumerate}
\end{definition}

Other more interesting definitions of subtraction can be obtained by imposing conditions
relating to discernibility on the \emph{target} $\overline{x}$ and/or \emph{source}
$\overline{y}$. 
\begin{enumerate}\renewcommand\labelenumi{\theenumi}
  \renewcommand{\theenumi}{(\roman{enumi})} 
\item {The condition 'a  sub-string will be removable from the target from the right if
and only if each letter in the sub-string is discernible from its successor or
predecessor' will be corresponded to $x\ominus_{1 \vee 2} y$.}    
\item {The condition 'a  sub-string will be removable from the target from the right if
and only if each letter in the sub-string is discernible from its successor and
predecessor' will be corresponded to $x\ominus_{12} y$.}    
\item {The condition 'a  sub-string will be removable from the target from the right if
and only if each letter in the sub-string is discernible from its successor' will be
corresponded to $x\ominus_{1} y$.}    
\item {The condition 'a  sub-string will be removable from the target from the right if
and only if each letter in the sub-string is discernible from its predecessor' will be
corresponded to $x\ominus_{2} y$.}    
\end{enumerate}

\begin{definition}
For any $x,\,y$, $x\preceq y$ if and only if
$x$ is obtainable from $y$ in Meta-R by a finite number of recursive applications of
$\varrho_{k}$ and $\iota_{k}$ operations for different values of $k$. 
\end{definition}

\begin{definition}
For any $x,\,y$, \[x\leq_{\oplus} y \,\,\mathrm{iff}\,\, (\exists z)\,(x\oplus z = y)\,
\vee\,(z\oplus x = y). \]
\end{definition}

\begin{definition}
For any $x,\,y$, \[x\leq_{\odot} y \,\,\mathrm{iff}\,\, (\exists z)\,(x\odot z = y) .\]
\end{definition}

\begin{definition}
For any $x,\,y$, \[x\leq_{\otimes} y \,\,\mathrm{iff}\,\, (\exists z)\,(x\otimes z = y) .\]
\end{definition}

\begin{definition}
For any $x,\,y$, $x\,\leq_{p}\,y$ in Meta-R if and only if $\zeta(x), \zeta(y)$ and
$\nu(x)\leq \nu(y)$ in Meta-C ($\leq$ being the usual total order on integers).
\end{definition}

\begin{definition}
For any $x,\,y$, $ x\,\trianglelefteq\, y$ if and only if $\nu(x)\leq \nu(y)$ in Meta-C
($\leq$ being the usual total order on integers). 
\end{definition}

\begin{definition}
For any $x,\,y$, $ x\,\sqsubseteq\, y$ in Meta-R if and only if $x$ is obtainable from $y$
in Meta-R by a finite number of recursive applications of $\varrho_{k}$ for different
values of $k$. 
\end{definition}

Using some of the different operations, many new partial algebras modelling the essence of
rough naturals at meta-R and meta-C are defined next. It should be noted that the
operations are well defined because IPC counts have an injective correspondence with
distributions of relatively discernible and indiscernible objects. 

\begin{definition}
By a \emph{Rough IPC-Natural Algebra} (RIPCNA) will be meant a partial algebraic system of
the form \[S\,=\,\left\langle\underline{S},\zeta, \oplus , \odot ,\ominus_{1\vee 2} , 0,
1, (1,2,2,2,0,0) \right\rangle, \]  
with the set $\underline{S}$ being the set of rough naturals formed according to the IPC
schema and with the above defined operations according to the IPC schema. $1$ can be
treated as an abbreviation of $1_{1}$, numbers of the form $1_{1}2_{1}...k_{1}$ can also
be abbreviated by $k$. $0$ will be understood as the IPC count of the empty string.
$\zeta$ is a one place predicate for indicating the usual integers:  \[\zeta (x)
\,\;\mathrm{iff}\,\; (\exists k)\, x = 1_{1}2_{1}...k_{1}.\]
\end{definition}

\begin{theorem}
In a RIPCNA of the form $S\,=\,\left\langle\underline{S}, \oplus , \odot ,\ominus_{1\vee
2} , 0, 1 \right\rangle$  of type $(2,2,2,0,0)$ all of the following hold:

\begin{enumerate}
\item {$(x\oplus y)\oplus z = x\oplus (y\oplus z)$,}
\item {$x \oplus 0 = x = 0 \oplus x$,}
\item {$x \oplus x = x \odot 2$,}
\item {$(\forall x, y, z)(x\oplus y = x \oplus z \longrightarrow y = z)$,}
\item {$(\forall x, y, z)(x\oplus z = y \oplus z \longrightarrow  x = y)$,}
\item {$x\odot(y\odot z) = (x\odot y) \odot z $,}
\item {$(\forall x, y, z)(x\odot y = z\odot y \longrightarrow x =z)$,}
\item {$x\odot 1 = x$;  $x\ominus x =0 $,}

\item {$x\odot (y \oplus z) = (x\odot y)\oplus (x\odot z)$,}
\item {$(x \odot x =x \longrightarrow (x = 0) \vee (x = 1)) $,}
\item {$(\forall x, y, z)(x = y \longrightarrow  x\odot z = x\odot z) $,}
\item {$(\forall x, y, z)(x \oplus y =z \longrightarrow z\ominus y = x) $,}
\item {$(\forall x, y)(\zeta x, \zeta y \longrightarrow x\oplus y = y\oplus x,\,x\odot y
=y \odot x) $,}
\item {$(\forall x, y, z)(\zeta x, \zeta y, \zeta z \longrightarrow (x\oplus y)\odot z =
(x\odot z)\oplus (y \odot z)$,}
\item {$(\forall x, y, z)(\zeta x, \zeta y \longrightarrow  (x\oplus y)\odot z = (x\odot
z)\oplus (y\odot z) )$.}
\end{enumerate}
\end{theorem}

\begin{proof}
\begin{enumerate}
\item {$(x\oplus y)\oplus z = x\oplus (y\oplus z)$ holds because either side is the IPC
count of the string formed by placing $\overline{x}$, $\overline{y}$ and $\overline{z}$ in
succession.}
\item {$0$ by definition is the count of the empty string, so $x \oplus 0 = x = 0 \oplus
x$.}
\item {$2$ is the same as $1_{1} 2_{1}$ and $x\odot 2$ would be the IPC count of the
string formed by placing two copies of $x$ in $2$ places. So $x \oplus x = x \odot 2$}
\item {If $x\oplus y = x \oplus z$ holds then either side will be the IPC count of
$\overline{x\oplus y}$ and the count of the initial string $\overline{x}$ will be common
to both. If $y\neq z$, then $x\oplus y$ cannot be equal to $x\oplus z$, so $(\forall x, y,
z)(x\oplus y = x \oplus z \longrightarrow y = z)$}
\item {Proof of $(\forall x, y, z)(x\oplus z = y \oplus z \longrightarrow  x = y)$ is
similar to the above.}
\item {The length of the string $\overline{y\odot z}$ will be the same as the length of
$\overline{y}$ multiplied by the length of $\overline{z}$. So $\overline{x}$ will be
equally replicated during the IPC counting operation corresponding to either side of
$x\odot(y\odot z) = (x\odot y) \odot z $.}
\item {$(\forall x, y, z)(x\odot y = z\odot y \longrightarrow x =z)$ can be verified from
the strings associated and a contradiction argument.}
\item {$x\odot 1 = x$ is obvious. If $\overline{x}$ is removed from $\overline{x}$, then
the result would be an empty string. So $x\ominus x =0 $}

\item {The length of $\overline{y\oplus z}$ is the same as the length of $\overline{y}$
plus that of $\overline{z}$. So $x\odot (y \oplus z) = (x\odot y)\oplus (x\odot z)$.}
\item {$(x \odot x =x \leftrightarrow (x = 0) \vee (x = 1)) $ is easy to verify.}
\item {$(\forall x, y, z)(x = y \longrightarrow  x\odot z = x\odot z) $  is easy to prove,
but note that $x\oplus z = y\oplus z$ need not follow from $x = y$ in general.}
\item {$(\forall x, y, z)(x \oplus y =z \longrightarrow z\ominus y = x) $ follows from the
definition of $\ominus$.}
\item {$(\forall x, y)(\zeta x, \zeta y \longrightarrow x\oplus y = y\oplus x,\,x\odot y
=y \odot x) $ follows from the properties of integers.}
\item {$(\forall x, y, z)(\zeta x, \zeta y, \zeta z \longrightarrow (x\oplus y)\odot z =
(x\odot z)\oplus (y \odot z)$ follows from the properties of integers.}
\item {$(\forall x, y, z)(\zeta x, \zeta y \longrightarrow  (x\oplus y)\odot z = (x\odot
z)\oplus (y\odot z) )$ is again a statement about integers.}
\end{enumerate}
\end{proof}

\begin{theorem}
$\{x\,:\,\zeta(x),\,x\in S\}$ with the operations $\oplus,\odot, \ominus$ is an integral
domain of characteristic 0 that is also a unique factorization domain. In fact it is
isomorphic to $Z$.   
\end{theorem}

\begin{proof}
$\{x\,:\,\zeta(x),\,x\in S\}$ is in bijective correspondence with $Z$ by definition. The
restrictions of the operations $\oplus,\odot, \ominus$ to the domain coincide with usual
arithmetic operations on $Z$.  
\end{proof}

As $\odot$ uses a certain level of knowledge of the natural numbers that is correctly at
Meta-C, the question of admissibility of such in some contexts can be a matter of dispute.
Relative this understanding $\otimes$ would be a better suited operation that is at the
same time a proper generalisation of usual multiplication and $\odot$. 

\begin{definition}
By a \emph{Rough IPC Algebra} (RIPCA) will be meant a partial algebraic system of the form
\[S\,=\,\left\langle\underline{S},\zeta, \oplus, \otimes ,\ominus_{1 \vee 2} , 0, 1,
(1,2,2,2,0,0) \right\rangle ,\]  
with the set $\underline{S}$ being the set of rough naturals formed according to the IPC
schema and with the above defined operations according to the IPC schema. $1$ can be
treated as an abbreviation of $1_{1}$, numbers of the form $1_{1}2_{1}...k_{1}$ will also
be abbreviated by $k$. $0$ will be understood as the IPC count of the empty string.
$\zeta$ is a one place predicate for indicating the usual integers:  \[\zeta (x)
\,\;\mathrm{iff}\,\; (\exists k)\, x = 1_{1}2_{1}...k_{1} .\]
\end{definition}

Of the orders $\sqsubseteq , \preceq , \leq_{\odot} ,\leq_{\otimes}, \leq_{\oplus}$ and
$\trianglelefteq$, the last is basically a Meta-C order corresponding to the usual order
on integers. Among the other five, $\preceq$ is the strongest with respect to inclusion,
but not much can be said about comparisons of their naturality in Meta-R. So it seems best
to define ordered versions of the partial algebras as follows:

\begin{definition}
By a \emph{R-Ordered Rough IPC Algebra} (RORIPCA) will be meant a partial algebraic system
of the form \[S\,=\,\left\langle\underline{S},\zeta,\sqsubseteq , \preceq , \leq_{\otimes}
, \leq_{\oplus}, \oplus, \otimes, \ominus_{1 \vee 2}, 0, 1, (1,2, 2, 2, 2, 2,2,2,0,0)
\right\rangle, \]  
with the set $\underline{S}$ being the set of rough naturals formed according to the IPC
schema and with \[S\,=\,\left\langle\underline{S},\zeta, \oplus, \otimes, \ominus_{1 \vee
2} , 0, 1, (1,2,2,2,0,0) \right\rangle ,\] being a RIPCA.
\end{definition}

\begin{definition}
By a \emph{C-Ordered Rough IPC Algebra} (CORIPCA) will be meant a partial algebraic system
of the form \[S\,=\,\left\langle\underline{S},\zeta,\trianglelefteq , \oplus, \otimes,
\ominus_{1\vee 2} , 0, 1, (1,2,2,2,2,0,0) \right\rangle ,\]  
with the set $\underline{S}$ being the set of rough naturals formed according to the IPC
schema and with \[S\,=\,\left\langle\underline{S},\zeta, \oplus, \otimes, \ominus_{1\vee
2} , 0, 1, (1,2,2,2,0,0) \right\rangle ,\] being a RIPCA.
\end{definition}

\begin{definition}
By a \emph{Full Ordered Rough IPC Algebra} (FORIPCA) will be meant a partial algebraic
system of the form \[S\,=\,\left\langle\underline{S},\zeta,\trianglelefteq , \sqsubseteq ,
\preceq , \leq_{\otimes} ,\leq_{\odot} \leq_{\oplus}, \leq_{p},\oplus, \otimes, \odot
,\ominus_{1\vee 2},\ominus, 0, 1 \right\rangle, \] of type $(1, 2, 2, 2, 2, 2,  2,
2,2,2,2,2,2,0,0)$  
with the set $\underline{S}$ being the set of rough naturals formed according to the IPC
schema and with \[S\,=\,\left\langle\underline{S},\zeta, \oplus, \otimes, \ominus_{1\vee
2} , 0, 1, (1,2,2,2,0,0) \right\rangle ,\] being a RIPCA. 
\end{definition}

Analogously, the concepts of FORIPCNA, CORIPCNA and RORIPCNA can be defined. 

\begin{proposition}
In a FORIPCA $S$, $\trianglelefteq$ is a quasi order that satisfies \[\trianglelefteq
\,\cap\, \zeta^{2} \,=\, \leq .\] 
\end{proposition}

\begin{proof}
It is easy to find elements $x, y$ that violate possible antisymmetry of
$\trianglelefteq$. The restriction to $\zeta^{2}$ is the same as the intersection with
$Z^{2}$. So on $Z$, the usual order $\leq$ coincides with $\trianglelefteq$. Strictly
speaking this is a category theoretic result involving $S$ and $Z$.
\end{proof}

\begin{proposition}
In a FORIPCA $S$, $\preceq$ is a quasi order. The class of any element in $\{x\,;\,
\zeta(x)\, x\neq 0\}$ with respect to the equivalence induced by $\preceq$ includes
$\{x\,;\, \zeta(x)\, x\neq 0\}$. 
\end{proposition}

\begin{proof}
Transitivity holds as if $x$ is obtainable from $y$ by the $\varrho_{k}$, $eta_{k}$
operations for different $k$ and $y$ is obtainable from $z$ similarly, then $x$ would be
obtainable from $z$. 
 
If $\zeta(x), \zeta(y)$, then $\overline{x}$ and $\overline{y}$ would be strings of
discernibles and so $x,\, y$ would be transformable into each other by the $\varrho_{k}$,
$eta_{k}$ operations. Defining $x\sim y$ if and only if $x\preceq y$ and $y \preceq x$, it
can be seen that $\sim$ is an equivalence and through this all of the elements of
$\{x\,;\, \zeta(x)\, x\neq 0\}$ can be identified.
\end{proof}

\begin{theorem}
In a FORIPCA $S$, the following implications between the different orders hold:
\begin{enumerate}
\item {$(\forall x, y)(x\leq_{\odot} y\,\longrightarrow\, x\leq_{\oplus} y) $,}
\item {$(\forall x, y)(x\leq_{\oplus} y\,\longrightarrow\, x \sqsubseteq y ) $,}
\item {$(\forall x, y)(\zeta x , \zeta y \,\longrightarrow\, x \trianglelefteq y \, \vee
\,y\trianglelefteq x \vee x = y )$,}
\item {$(\forall x, y)(x \leq_{p} y \,\longrightarrow\, x \trianglelefteq y )$,}
\item {$(\forall x, y)(x\otimes y \leftrightarrow x\odot y) $,}
\item {For any $x,\,y$, $x\leq_{\oplus} y$ does not not imply that $x\preceq y$ and
neither does the converse implication hold.}
\end{enumerate}
\end{theorem}

\begin{proof}
Most of the statements can be proved from the corresponding definitions. $\leq_{p}$ is a
partial order that is coincides with the usual order when restricted to $Z$. Statement 5 is ensured by the existence of elements corresponding to all possible
relative discernibility-indiscernibility patterns. 

If $\overline{z}$ is a string of elements, some of which are indiscernible from
$\overline{x}$ and $y= x\oplus z$, then x may not be obtainable from $y$. Concrete
patterns can be constructed at different levels of complexity. The easiest case
corresponds to all elements of $\overline{z}$ being indiscernible from those of
$\overline{x}$. 
\end{proof}

A useful visual representation of the computing process for $a\,\otimes\,b $ is
illustrated below through a specific example. In the first line blank spaces are drawn for
the alphabets in $\overline{b}$. The curved arrows are used to indicate discernibility.
The second figure represents a gross view of the strings in $a\,\otimes\, b$.

\begin{equation*}
\setcounter{MaxMatrixCols}{6}
\begin{matrix}
\overline{a} & \overline{a}  &  &  & \overline{a} & \\
-\negthickspace - & -\negthickspace - & -\negthickspace - & -\negthickspace - &
-\negthickspace - & -\negthickspace - .\\
   & \circlearrowright &  &  &  \circlearrowright & \\
\end{matrix}
\end{equation*}

From the above, it can be deduced that $a\,\otimes\, b$ is equivalent to the count of 

\begin{equation*}
\setcounter{MaxMatrixCols}{3}
\begin{matrix}
\overline{a} & \overline{a} & \overline{a} \\
-\negthickspace -& -\negthickspace - & -\negthickspace - .\\
\end{matrix}
\end{equation*}

Key properties of the relatively difficult operation $\otimes$ are considered in the
following theorems: 
\begin{theorem}
In a FORIPCA $S$, all of the following hold:
\begin{enumerate}
\item {$(\forall a, b, c )(a\,\otimes\, b)\,\otimes \,c \,\trianglelefteq\, a\,\otimes\,
(b\,\otimes\,c) $,}
\item {$(\forall a, b, c ) (\zeta(b)\,\longrightarrow\,(a\,\otimes\, b)\,\otimes \,c \,=\,
a\,\otimes\, (b\,\otimes\,c) )$,}
\item {$(\forall a, b) (a\,\neq\,0,\,\,a\,\otimes\, b\,=\,a\,\odot\, b\,\leftrightarrow\,
b\,\ominus_{12}\,b\,=\,0) $.}
\end{enumerate}
\end{theorem}
\begin{proof}
\begin{enumerate}
\item {The length of $a\otimes b$ is determined by the distribution of indiscernible pairs
in $b$. The proof can be done by considering the different cases of the $\otimes$
multiplication in the left and right side of the inequality corresponding to the relation
of the first and last element of the strings corresponding to $a,\, b$ and $a\otimes b$. }
\item {If $\zeta (b)$, then the string corresponding to it will have mutually discernible
elements of the length of $b$. In $(a\otimes b)\otimes c$, $a\otimes b$ will be repeated
the same number of times as $b$ in $b\otimes c$. So the equality will hold.}
\item {If $a\,\otimes\,b\,=\,a\,\odot\,b$, then as $a$ is not $0$, it is necessary that
each object in $\overline{b}$ be discernible from its immediate predecessor (that is for
every admissible $k$, $(\overline{b}_k , \overline{b}_{k+1})\,\notin\,\approx $). So
$b\,\ominus_{12}\,b\, =\,0$ must hold (irrespective of the length of $b$). Note that for
an arbitrary element $x$, it need not happen that $x\,\ominus_{12}\,x\,=\,0$ in general. 

For the converse, note that if $b\,\ominus_{12}\,b\,=\,0$, then it will be necessary that
for every admissible $k$, $(\overline{b}_k , \overline{b}_{k+1})\,\notin\,\approx $. Both
the multiplications of $a$ with such a $b$ will be equal. 
}
\end{enumerate}
\end{proof}

In the next theorem, the compatibility of the different orders defined are considered: 

\begin{theorem}
In a FORIPCA $S$, the following hold:
\begin{enumerate}
\item {$(\forall a, b, c, e)(a\trianglelefteq b,\,c\trianglelefteq e
\,\longrightarrow\,(c\oplus a)\trianglelefteq (e\oplus b),\,(a\odot c)\trianglelefteq
(b\odot e))$,}
\item {$(\forall a, b, c, e)(a\leq_{\oplus} b, \, c\leq_{\oplus} e
\,\longrightarrow\,a\oplus c\,\trianglelefteq\,b\oplus e )$,}
\item {$(\forall a, b, c)(a\leq_{\oplus} b \,\longrightarrow\, c\oplus a \leq_{\oplus}
c\oplus b,\,\mathrm{or}\,\,a\oplus c\leq_{\oplus} b\oplus c )$,}
\item {$(\forall a, b)(a\leq_{\oplus} b, a\sqsubseteq b \,\longrightarrow\,
b\ominus_{1\vee 2} a \sqsubseteq b )$,}
\item {$(\forall a, b, c, e)(a\leq_{\odot} b, c\leq_{\odot} e \,\longrightarrow\,a\odot
c\trianglelefteq b\odot e ,\,a\odot c \sqsubseteq b\odot e  )$,}
\item {$(\forall a, b, c, e)(a\sqsubseteq b, c\subseteq e \,\longrightarrow\, a\oplus
c\sqsubseteq b\oplus e,\, a\odot c\sqsubseteq b\odot e,)$.}
\end{enumerate}
\end{theorem}

\begin{proof}
\begin{enumerate} 
\item {$\trianglelefteq$ corresponds to the meta-C interpretation by the length of
associated. So the $\odot$ and $\oplus$ part of the implication should be obvious. However
as $c\trianglelefteq e$ does not imply that the number of objects in $c$ that are
discernible from their predecessor are less than the corresponding number in $e$. So
$\otimes$ will not preserve $\trianglelefteq$. }
\item {In general if $a\leq_{\oplus} b, \, c\leq_{\oplus} e$, then it need not happen that
$a\oplus c\,\leq_{\oplus} \, b\oplus e$, because of the different possible relations
between the objects at the terminal and initial position of $\oplus$. But the gross length
of $\overline{a\oplus c}$ will be $\nu (a) + \nu (c)$. So the implication holds.}
\item { Follows from the definition of $\leq_{\oplus} $. }
\item { If $a\leq_{\oplus} b $, then there exists a $c$ such that $a\oplus c = b$ or
$c\oplus a = b$. In either  case, if $\overline{a}$ is obtainable from $\overline{b}$ by a
finite number of recursive applications of $\rho_{k}$, then it is necessary that these
operations must have been applied at one end of $\overline{b}$. This causes
$b\ominus_{1\vee 2} a \sqsubseteq b$.  }
\item {In $\overline{a\odot c}$ and $\overline{b\odot e}$, $\overline{a}$ is repeated
$\nu(c)$ times and $b$ is repeated $\nu{e}$ times. So the length of $\overline{a\odot c}$
will be less than that of $\overline{b\odot e}$. Further as  $a\leq_{\odot} b$ and
$c\leq_{\odot} e $, it will be possible to obtain $\overline{a\odot c}$ from
$\overline{b\odot e}$ through a finite number of applications of $\rho_{k}$ for different
values of $k$. So the result holds.}
\item {This follows from definition.}
\end{enumerate}
\end{proof}

Interestingly it is not possible to define a unary negation operator from any of the
subtraction-like  $\ominus$ operations and $\oplus$ in a consistent way. So the concept of
negative elements does not generalise well to the present contexts. The extent to which a
consistent definition is possible will be of natural interest.

\section{Further Directions: Conclusion}

The broad classes of problems that will be part of future work fall under:
\begin{enumerate}\renewcommand\labelenumi{\theenumi}
  \renewcommand{\theenumi}{(\roman{enumi})} 
 \item {Improvement of the representation of different classes of counts. One class of
questions also relate to using morphisms or automorphisms in a more streamlined way.}
 \item {Extension of the algebraic approach (and the concept of rough natural numbers) of
the last section to non IPC cases.  }
 \item {The contamination problem.}
 \item {Description of other semantics of general \textsf{RST} in terms of counts. This is
the basic program of representing all types of general rough semantics in terms of
counts.}
 \item {Can the division operation be eliminated in the general measures proposed? In
other words, a more natural extensions of fractions (rough rationals) would be of much
interest. Obviously this is part of the representation problem for counts.}
 \item {How do algorithms for reduct computation get affected by the generalised
measures?}
 \end{enumerate}
Apart from these a wide variety of combinatorial questions would be of natural interest.

In this research paper, the axiomatic theory of granules and granulation developed by the
present author has been extended to cover most types of general \textsf{RST} and new
methods of counting collections of well defined and indiscernible objects have been
integrated with it. These new methods of counting have been shown to be applicable to the
extension of fundamental measures of \textsf{RST}, rough inclusion measures and
consistency degrees of knowledge. The redefined measures possess more information than the
original measures and have more realistic orientation with respect to counting. The
connections with semantic domains, that are often never explicitly formulated, has been
brought into sharp focus through the approach.

The concept of rough naturals in the IPC perspective has also been developed in this
paper. The new program centred around the contamination problem proposed in this paper can
also be found in another forthcoming paper by the present author on axiomatic theory of
granules for \textsf{RST}. Here the direction is made far more clearer through the
integration with rough naturals.

By the mathematics of vagueness, I do not mean a blind transfer of the results of the
mathematics of exact contexts to inexact contexts. It is intended to incorporate vagueness
in more natural and amenable ways in the light of the contamination problem. An essential
part of this is achieved in this research paper.  

\section*{Acknowledgement}
I would like to extend my thanks to the referee(s) for useful comments and pointing
lapses in the use of the English language.

\bibliographystyle{splncs.bst}
\bibliography{newsem99999.bib}

\begin{thebibliography}{10}

\bibitem{Sk08}
Skowron, A., Peters, J.:
\newblock Rough-granular computing.
\newblock In Pedrycz, W., Skowron, A., Kreinovich, V., eds.: Granular
  Computing.
\newblock John Wiley \& Sons, Ltd., Chichester, UK (2008)  225--328

\bibitem{Paw94}
Pawlak, Z.:
\newblock An inquiry into vagueness and uncertainty.
\newblock PAS Institute of Computer Science \textbf{29/94} (1994)

\bibitem{PPM2}
Pagliani, P., Chakraborty, M.:
\newblock A Geometry of Approximation: Rough Set Theory: Logic, Algebra and
  Topology of Conceptual Patterns.
\newblock Springer, Berlin (2008)

\bibitem{Baz06}
Bazan, J., Skowron, A., Swiniarski, R.:
\newblock Rough sets and vague concept approximation: From sample approximation
  to adaptive learning.
\newblock Trans. on Rough Sets \textbf{V, LNCS 4100} (2006)  39--62

\bibitem{Sk05}
Skowron, A.:
\newblock Rough sets and vague concepts.
\newblock Fund. Inform. \textbf{64} (2005)  417--431

\bibitem{AM909}
Mani, A.:
\newblock Towards logics of some rough perspectives of knowledge.
\newblock In Suraj, Z., Skowron, A., eds.: Intelligent Systems Reference
  Library dedicated to the memory of Prof. Pawlak,.
\newblock In Press (2011)  25pp

\bibitem{Ba03}
Banerjee, M., Chakraborty, M.K.:
\newblock Foundations of vagueness: {A} categorytheoretic approach.
\newblock Electronic Notes in Theoretical Computer Science \textbf{82} (2003)
  10--19

\bibitem{We07}
Weiner, J.:
\newblock Science and semantics: The case of vagueness and supervaluation.
\newblock Pacific Philosophical Quarterly \textbf{88} (2007)  355--374

\bibitem{Sh06}
Shapiro, S.:
\newblock Vagueness in Context.
\newblock Oxford University Press, Oxford, UK (2006)

\bibitem{Ke00}
Keefe, R.:
\newblock Theories of Vagueness.
\newblock Cambridge University Press, Cambridge, UK (2000)

\bibitem{Ke99}
Keefe, R., Smith, P.:
\newblock Vagueness: {A} Reader.
\newblock MIT Press, Cambridge, MA (1999)

\bibitem{Ru23}
Russell, B.:
\newblock Vagueness.
\newblock The Australasian Journal of Psychology and Philosophy \textbf{1}
  (1923)  84--92

\bibitem{Paw97}
Pawlak, Z.:
\newblock Vagueness - a rough set view.
\newblock In Mycielski, J., Rozenberg, G., Salomaa, A., eds.: Structures in
  Logic and Computer Science, LNCS 1261.
\newblock Springer, Berlin (1997)  106--117

\bibitem{Bo08}
Bonikowski, Z., Wybraniec-Skardowska, U.:
\newblock Vagueness and roughness.
\newblock Trans. on Rough Sets \textbf{IX, LNCS 5390} (2008)  1--13

\bibitem{DP80}
Dubois, D., Prade, H.:
\newblock Fuzzy Sets and Systems.
\newblock Academic Press, New York (1980)

\bibitem{LP02}
Polkowski, L.:
\newblock Rough Sets: Mathematical Foundations.
\newblock Physica-Verlag, Heidelberg (2002)

\bibitem{ZPB}
Pawlak, Z.:
\newblock Rough Sets: Theoretical Aspects of Reasoning About Data.
\newblock Kluwer Academic Publishers, Dodrecht (1991)

\bibitem{AM99}
Mani, A.:
\newblock Choice inclusive general rough semantics.
\newblock Information Sciences \textbf{181} (2011)  1097--1115

\bibitem{ZP5}
Pawlak, Z.:
\newblock Some issues in rough sets.
\newblock In Skowron, A., Peters, J.F., eds.: Transactions on Rough Sets- I.
  Volume 3100.
\newblock Springer Verlag (2004)  1--58

\bibitem{PS3}
Polkowski, L., Skowron, A.:
\newblock Rough mereology: A new paradigm for approximate reasoning.
\newblock Internat. J. Appr. Reasoning \textbf{15} (1996)  333--365

\bibitem{TYL}
Lin, T.Y.:
\newblock Granular computing -1: The concept of granulation and its formal
  model.
\newblock Int. J. Granular Computing, Rough Sets and Int Systems \textbf{1}
  (2009)  21--42

\bibitem{YY5}
Yao, Y.:
\newblock The art of granular computing.
\newblock In Kryszkiewicz, M.,  et~al., eds.: RSEISP'2007, LNAI 4585, Springer
  Verlag (2007)  101--112

\bibitem{WW}
Zhang, M.W.X., Wu, W.Z., Li, T.J.:
\newblock Granulation.
\newblock In An, A., Yao, J.T.,  et~al., eds.: RSFDGrC 2007, LNAI 4481.
\newblock Springer-Verlag (2007)  93--100

\bibitem{SW3}
Wasilewski, P., Slezak, D.:
\newblock Foundations of rough sets from vagueness perspective.
\newblock In Hassanien, A.,  et~al., eds.: Rough Computing: Theories,
  Technologies and Applications. Information Science Reference.
\newblock IGI, Global (2008)  1--37

\bibitem{KCM}
Keet, C.M.:
\newblock A Formal Theory of Granules - Phd Thesis.
\newblock PhD thesis, Fac of Comp.Sci., Free University of Bozen (2008)

\bibitem{PL}
Polkowski, L.:
\newblock Rough mereology as a link between rough and fuzzy set theories.
\newblock In Peters, J.F., Skowron, A., eds.: Transactions on Rough Sets.
  Volume LNCS 3135.
\newblock Springer Verlag (2004)  253--277

\bibitem{GT}
Takeuti, G.:
\newblock Proof Theory. 2nd edn.
\newblock North Holland, Amsterdam (1987)

\bibitem{FJ}
Font, J.M., Jansana, R.:
\newblock A General Algebraic Semantics for Sentential Logics. Volume~7.
\newblock Association of Symbolic Logic (2009)

\bibitem{BC1}
Banerjee, M., Chakraborty, M.K.:
\newblock Rough sets through algebraic logic.
\newblock Fundamenta Informaticae \textbf{28} (1996)  211--221

\bibitem{DU}
Duntsch, I.:
\newblock Rough sets and algebras of relations.
\newblock In Orlowska, E., ed.: Incomplete Information and Rough Set Analysis,
  Physica, Heidelberg (1998)  109--119

\bibitem{AM3}
Mani, A.:
\newblock Super rough semantics.
\newblock Fundamenta Informaticae \textbf{65} (2005)  249--261

\bibitem{ZP6}
Pawlak, Z.:
\newblock Rough sets.
\newblock Internat. J. of Computing and Information Sciences \textbf{18} (1982)
   341--356

\bibitem{SS1}
Skowron, A., Stepaniuk, O.:
\newblock Tolerance approximation spaces.
\newblock Fundamenta Informaticae \textbf{27} (1996)  245--253

\bibitem{KM}
Komorowski, J., Pawlak, Z., Polkowski, L., Skowron, A.:
\newblock Rough sets -- a tutorial.
\newblock In Pal, S.K., Skowron, A., eds.: Rough Fuzzy Hybridization.
\newblock Springer Verlag (1999)  3--98

\bibitem{CG98}
Cattaneo, G.:
\newblock Abstract approximation spaces for rough set theory.
\newblock In Polkowski, L., Skowron, A., eds.: Rough Sets in Knowledge
  Discovery 2.
\newblock Physica Heidelberg (1998)  59--98

\bibitem{KB98}
Konikowska, B.:
\newblock A logic for reasoning with similarity.
\newblock In Orlowska, E., ed.: Incomplete Information and Rough Set Analysis,
  Physica Verlag, Heidelberg (1998)  461--490

\bibitem{IM}
Inuiguchi, M.:
\newblock Generalisation of rough sets and rule extraction.
\newblock In Peters, J.F., Skowron, A., eds.: Transactions on Rough Sets-1.
  Volume LNCS-3100.
\newblock Springer Verlag (2004)  96--116

\bibitem{SW}
Slezak, D., Wasilewski, P.:
\newblock Granular sets - foundations and case study of tolerance spaces.
\newblock In An, A., Stefanowski, J., Ramanna, S., Butz, C.J., Pedrycz, W.,
  Wang, G., eds.: RSFDGrC 2007, LNCS. Volume 4482.
\newblock Springer (2007)  435--442

\bibitem{AM105}
Mani, A.:
\newblock Algebraic semantics of similarity-based bitten rough set theory.
\newblock Fundamenta Informaticae \textbf{97} (2009)  177--197

\bibitem{PO}
Pomykala, J.A.:
\newblock Approximation, similarity and rough constructions.
\newblock Technical Report CT-93-07, ILLC, Univ of Amsterdam (1993)

\bibitem{CC}
Cattaneo, G., Ciucci, D.:
\newblock Algebras for rough sets and fuzzy logics.
\newblock In Skowron, A., Peters, J.F., eds.: Transactions on Rough Sets II.
  Volume~2.
\newblock Springer Verlag (2004)  208--252

\bibitem{WZ}
Zakowski, W.:
\newblock Approximation in the space $(u, \pi)$.
\newblock Demonstration Math \textbf{XVI} (1983)  761--769

\bibitem{YY9}
Yao, Y.:
\newblock Relational interpretation of neighbourhood operators and rough set
  approximation operators.
\newblock Information Sciences (1998)  239--259

\bibitem{LTJ}
Li, T.J.:
\newblock Rough approximation operators in covering approximation spaces.
\newblock In Greco, S.,  et~al., eds.: RSCTC 2006. LNCS (LNAI) 4259, Springer
  Verlag (2006)  174--182

\bibitem{ZW3}
Zhu, W.:
\newblock Relationship between general rough set based on binary relation and
  covering.
\newblock Information Sciences \textbf{179} (2009)  210--225

\bibitem{CP4}
Samanta, P., Chakraborty, M.K.:
\newblock Covering based approaches to rough sets and implication lattices.
\newblock In Sakai, H.,  et~al., eds.: RSFDGrC '2009, LNAI 5908, Springer
  (2009)  127--134

\bibitem{AM24}
Mani, A.:
\newblock Esoteric rough set theory-algebraic semantics of a generalized vprs
  and vprfs.
\newblock In Skowron, A., Peters, J.F., eds.: Transactions on Rough Sets VIII.
  Volume LNCS 5084.
\newblock Springer Verlag (2008)  182--231

\bibitem{AM960}
Mani, A.:
\newblock Towards an algebraic approach for cover based rough semantics and
  combinations of approximation spaces.
\newblock In Sakai, H.,  et~al., eds.: RSFDGrC 2009. Volume LNAI 5908.,
  Springer-Verlag (2009)  77--84

\bibitem{AM699}
Mani, A.:
\newblock Integrated dialectical logics for relativised general rough set
  theory.
\newblock In: Internat. Conf. on Rough Sets, Fuzzy Sets and Soft Computing,
  Agartala, India, http://arxiv.org/abs/0909.4876 (2009)  6pp (Refereed)

\bibitem{MD}
Mundici, D.:
\newblock Generalization of abstract model theory.
\newblock Fundamenta Math \textbf{124} (1984)  1--25

\bibitem{BY}
Banerjee, M., Yao, Y.:
\newblock Categorical basis for granular computing.
\newblock In Stefanowski, J., Ramanna, S., Butz, C.J., Pedrycz, W., Wang, G.,
  eds.: RSFDGrC 2007, LNAI.
\newblock Springer (2007)  427--434

\bibitem{BK3}
Banerjee, M., Khan, M.A.:
\newblock Propositional logics for rough set theory.
\newblock In: Transactions on Rough Sets VI, LNCS 4374, Springer Verlag (2007)
  1--25

\bibitem{BC2}
Banerjee, M., Chakraborty, M.K.:
\newblock Algebras from rough sets -- an overview.
\newblock In Pal, S.K., {et. al}, eds.: Rough-Neural Computing.
\newblock Springer Verlag (2004)  157--184

\bibitem{YY3}
Yao, Y.:
\newblock A partition model for granular computing.
\newblock In Skowron, A., Peters, J.F., eds.: Transactions on Rough Sets-1,
  LNCS 3100.
\newblock Springer (2004)  232--253

\bibitem{PPM}
Pagliani, P., Chakraborty, M.K.:
\newblock Formal topology and information systems.
\newblock In: Transactions on Rough Sets VI, Springer (2007)  253--297

\bibitem{AM1005}
Mani, A.:
\newblock Dialectics of counting and measures of rough set theory.
\newblock In: IEEE Proceedings of NCESCT'2011, Pune, Feb-1-3, Arxiv:1102.2558
  (2011)  17pp

\bibitem{AF}
Arnold, D., Fowler, K.:
\newblock Nefarious numbers.
\newblock Notices Amer. Math. Soc. \textbf{58} (2011)

\bibitem{Fine75}
Fine, K.:
\newblock Vagueness, truth and logic.
\newblock Synthese \textbf{30} (1975)  265--300

\bibitem{Dummet}
Dummet, M.:
\newblock The Logical Basis of Metaphysics.
\newblock Harvard University Press (1991)

\bibitem{PSI}
Simon, P.:
\newblock Parts - A Study in Ontology.
\newblock Oxford Univ. Press (1987)

\bibitem{LDK}
Lewis, D.K.:
\newblock Parts of Classes.
\newblock Oxford; Basil Blackwell (1991)

\bibitem{AV}
Varzi, A.:
\newblock Parts, wholes and part-whole relations: The prospects of
  mereotopology.
\newblock Data and Knowledge Engineering \textbf{20} (1996)  259--286

\bibitem{LP4}
Polkowski, L.
\newblock In: Rough Neural Computation Model Based on Rough Mereology. Springer
  Verlag (2004)  85--108

\bibitem{GK}
Grzegorczyk, A.:
\newblock The systems of lesniewski in relation to contemporary logical
  research.
\newblock Studia Logica \textbf{3} (1955)  77--95

\bibitem{PLS}
Polkowski, L., Polkowska, S.M.:
\newblock Reasoning about concepts by rough mereological logics.
\newblock In Wang, G.,  et~al., eds.: RSKT 2008, LNAI 5009, Springer (2008)
  197--204

\bibitem{UR}
Urbaniak, R.:
\newblock Lesniewski's Systems of Logic and Mereology;History and
  Re-evaluation.
\newblock PhD thesis, Department of Philosophy, Univ of Calgary (2008)

\bibitem{VP9}
Vopenka, P.:
\newblock Mathematics in the Alternative Set Theory.
\newblock Teubner-Verlag (1979)

\bibitem{VH}
Vopenka, P., Hajek, P.:
\newblock The Theory of Semisets.
\newblock Academia, Prague (1972)

\bibitem{IT2}
Iwinski, T.B.:
\newblock Rough orders and rough concepts.
\newblock Bull. Pol. Acad. Sci (Math) \textbf{(3--4)} (1988)  187--192

\bibitem{CC5}
Cattaneo, G., Ciucci, D.:
\newblock Lattices with interior and closure operators and abstract
  approximation spaces.
\newblock In Peters, J.F.,  et~al., eds.: Transactions on Rough Sets X, LNCS
  5656.
\newblock Springer (2009)  67--116

\bibitem{PD}
Perry, B., Dockett, S.:
\newblock Young children's access to powerful mathematical ideas.
\newblock In English, L.,  et~al., eds.: Handbook of International Research in
  Mathematics Education.
\newblock LEA (2002)  81--112

\bibitem{SM}
Satyanarayana, M.:
\newblock Positively Ordered Semigroups. Volume~42 of Lecture Notes in Pure and
  Applied Mathematics.
\newblock Marcel Dekker Inc. (1979)

\bibitem{AM90}
Mani, A.:
\newblock V-perspectives, pseudo-natural number systems and partial orders.
\newblock Glasnik Math. \textbf{37} (2002)  245--257

\bibitem{CD3}
Ciucci, D.:
\newblock Approximation algebra and framework.
\newblock Fundamenta Informaticae \textbf{94} (2009)  147--161

\bibitem{SPD}
Sen, D., Pal, S.K.:
\newblock Generalized rough sets, entropy and image ambiguity measures.
\newblock IEEE Trans. -PartB: Cybernetics \textbf{39} (2008)  117--128

\bibitem{PP4}
Pagliani, P.:
\newblock Pretopologies and dynamic spaces.
\newblock Fundamenta Informaticae \textbf{59} (2004)  221--239

\bibitem{AQ1}
Khan, M.A., Banerjee, M.:
\newblock Formal reasoning with rough sets in multiple-source approximation
  spaces.
\newblock Internat. J. Approximate Reasoning \textbf{49} (2008)  466--477

\bibitem{PN6}
Novotny, M., Pawlak, Z.:
\newblock Characterization of rough top and bottom equalities.
\newblock Bull. Pol. Acad. Sci (Math) \textbf{33} (1985)  91--97

\bibitem{WH}
Hodges, W.:
\newblock Elementary predicate logic.
\newblock In Gabbay, D., Guenthner, F., eds.: Handbook of Philosophical Logic.
\newblock Academic, Kluwer (2004)  1--130

\bibitem{ZW}
Ziarko, W.:
\newblock Variable precision rough set model.
\newblock J. of Computer and System Sciences \textbf{46} (1993)  39--59

\bibitem{RR}
Rolka, A.M., Rolka, L.:
\newblock Variable precision fuzzy rough sets.
\newblock In Peters, J.F., Skowron, A., eds.: Transactions in Rough Sets-1.
  Volume LNCS 3100.
\newblock Springer Verlag (2004)  144--160

\bibitem{MCE}
Carrara, M., Martino, E.:
\newblock On the ontological commitment of mereology.
\newblock Rev. Symb. Logic \textbf{2} (2009)  164--174

\bibitem{QL}
Qian, Y.H., Liang, J.Y.:
\newblock Rough set method based on multi- granulations.
\newblock In: Proc. 5th IEEE Conf. Cognitive Informatics. Volume~1. (2006)
  297--304

\bibitem{QLY}
Qian, Y., Liang, J.Y., Yao, Y.Y., Dang, C.Y.:
\newblock Mgrs: A multi granulation rough set.
\newblock Information Science \textbf{180} (2010)  949--970

\bibitem{QLY2}
Qian, Y.H., Liang, J.Y., Dang, C.Y.:
\newblock Incomplete mutigranulation rough set.
\newblock IEEE Transactions on Systems, Man and Cybernetics, Part A \textbf{20}
  (2010)  420--430

\bibitem{SLV}
Slowinski, R., Vanderpooten, D.:
\newblock A generalized definition of rough approximations based on similarity.
\newblock IEEE Trans on Knowledge and Data Engg. \textbf{12} (2000)  331--336

\bibitem{AM69}
Mani, A.:
\newblock Meaning, choice and similarity based rough set theory.
\newblock Internat. Conf. Logic and Appl., Jan'2009 Chennai;(Refereed),
  http://arxiv.org/abs/0905.1352 (2009)  1--12

\bibitem{TYL3}
Lin, T.Y.:
\newblock Neighbourhood systems- applications to qualitative fuzzy and rough
  sets.
\newblock In Wang, P.P.,  et~al., eds.: Advances in Machine Intelligence and
  Soft Computing' Duke University, Durham'1997. (1997)  132--155

\bibitem{ZHU3}
Zhu, W.:
\newblock Topological approaches to covering rough sets.
\newblock Information Sciences \textbf{177} (2007)  1499--1508

\bibitem{CZ}
Chajda, I., Niederle, J., Zelinka, B.:
\newblock On existence conditions for compatible tolerances.
\newblock Czech. Math. J \textbf{26} (1976)  304--311

\bibitem{JJ}
Jarvinen, J.:
\newblock Lattice theory for rough sets.
\newblock In Peters, J.F.,  et~al., eds.: Transactions on Rough Sets VI. Volume
  LNCS 4374.
\newblock Springer Verlag (2007)  400--498

\bibitem{OR}
Orlowska, E., Rewitz, I.:
\newblock Discrete duality and its application to reasoning with incomplete
  information.
\newblock In Kryszkiewicz, M.,  et~al., eds.: RSEISP'2007, LNAI 4585, Springer
  Verlag (2007)  51--56

\bibitem{OE3}
Orlowska, E.:
\newblock Kripke semantics for knowledge representation logics.
\newblock Studia Logica \textbf{XLIX} (1990)  255--272

\bibitem{CH}
Hartonas, C.:
\newblock Duality for lattice ordered algebras and for normal algebraizable
  logics.
\newblock Studia Logica \textbf{58} (1997)  403--450

\bibitem{CHP}
Chakraborty, M.K., Samanta, P.:
\newblock Consistency degree between knowledges.
\newblock In Kryszkiewicz, M.,  et~al., eds.: RSEISP'2007. Volume LNAI 4583.,
  Springer Verlag (2007)  133--141

\bibitem{AG3}
Gomolinska, A.:
\newblock On certain rough inclusion functions.
\newblock In Peters, J.F.,  et~al., eds.: Transactions on Rough Sets IX, LNCS
  5390.
\newblock Springer Verlag (2008)  35--55

\bibitem{BCB}
Banerjee, M., Chakraborty, M.K., Bunder, M.:
\newblock Some rough consequence logics and their interrelations.
\newblock In Skowron, A., Peters, J.F., eds.: Transactions on Rough Sets VIII.
  Volume LNCS 5084., Springer Verlag (2008)  1--20

\bibitem{YYFR}
Yao, Y.:
\newblock Semantics of fuzzy sets in rough set theory.
\newblock In Skowron, A., Peters, J.F., eds.: Transactions on Rough Sets II.
  Number LNCS 3135.
\newblock Springer Verlag (2005)  297--318

\bibitem{RV}
Ramik, J., Vlach, M.:
\newblock A non-controversial definition of fuzzy sets.
\newblock In Skowron, A., Peters, J.F.,  et~al., eds.: Transactions on Rough
  Sets II. Volume~II., Springer (2004)  201--207

\end{thebibliography}

\end{document}